\documentclass[12pt, reqno]{amsart}

\usepackage{amsmath,amsfonts, amsthm, amssymb, color, cite}
\usepackage{color}
\usepackage{fancyhdr}
\usepackage{graphicx}
\usepackage{epstopdf}
\usepackage{lastpage}
\usepackage{refcheck}
\newtheorem{thm}{Theorem}[section]
\newtheorem{lma}{Lemma}[section]
\newtheorem{pro}{Proposition}[section]

\theoremstyle{definition}

\theoremstyle{remark}

\numberwithin{equation}{section}
\allowdisplaybreaks

\def\f{\frac}

\def\hf1{^\f{1}{1-\xi^2}}

\def\be{\begin{equation}}
\def\en{\end{equation}}
\def\bs{\begin{split}}
\def\es{\end{split}}
\def\ba{\begin{align}}
\def\ea{\end{align}}

\author[Lin Chang]{Lin Chang}
\address{School
of Mathematical   Sciences, Beihang University, Beijing, China}
\email{ changlin23@buaa.edu.cn}
\renewcommand{\fancyhead}{}
\title{ C\MakeLowercase{onvergence} \MakeLowercase{rate} \MakeLowercase{toward}  \MakeLowercase{shock} \MakeLowercase{wave}  \MakeLowercase{under} \MakeLowercase{periodic} \MakeLowercase{perturbation} \MakeLowercase{for} \MakeLowercase{generalized} K\MakeLowercase{orteweg}-\MakeLowercase{de} V\MakeLowercase{ries}-B\MakeLowercase{urgers} \MakeLowercase{equation}}
\date{\today}
\begin{document}
\begin{abstract}
In this paper, a viscous shock wave under space-periodic perturbation   of   generalized Korteweg-de Vries-Burgers equation
is investigated. It is shown that  if the initial periodic perturbation around the viscous shock wave is small, then the solution  time asymptotically tends to a viscous shock wave  with a shift partially determined by  the periodic oscillations. Moreover the exponential time decay rate toward the viscous shock wave is also obtained for some certain   perturbations.
\end{abstract}
\maketitle
\color{black}\section{Introduction}
We consider  generalized Korteweg-de Vries-Burgers (KdV-Burgers) equation:
\begin{align}\label{1.1}
u_{t}+ f(u)_{x}+\textcolor{blue}{\mu} u_{xxx}{\textcolor{red}{=}}\textcolor{blue}{\gamma} u_{xx},    &\ x\in \mathbb{R},\ t>0,
\end{align}
where  the flux $f(u)\in C^3$  is  strictly convex, $\textcolor{blue}{\mu}>0$ the dispersive coefficient,  $\textcolor{blue}{\gamma}>0$   the viscosity.
\
When $\textcolor{blue}{\mu}{\textcolor{red}{=}}0$, the  equation  \eqref{1.1}  becomes the famous Burgers equation, which admits viscous shock wave solution  $\textcolor{blue}{\phi}(x-st)$ with shock propagation speed $s$. The stability of viscous shock for the Burgers equation has been extensively studied, see \cite{km1985,XYY2021,O}. When $\textcolor{blue}{\mu}>0$,    the KdV-Burgers equation \eqref{1.1} still  admits the  viscous shock wave solution as $\textcolor{blue}{\gamma}\gg 1$, that is  the viscosity plays the main role, see \cite{GH1967,BS1985}. Later, Bona-Rajopadhye-Schonbek     \cite{BR1994} showed that this shock wave is asymptotically stable in the case of $   f{\textcolor{red}{=}}\frac{1}{2} u^{2}$ provided that the perturbation is small. Moreover, the time decay rate was obtained in \cite{R1995,NR1998,PL2003}. The exponential time decay rate was further obtained in Yin-Zhao-Zhou \cite{YZZ2009}   provided that the initial values converge to  some constants  exponentially at the far field. We refer to \cite{q5,q4,q1,q2,q3} for more interesting works on KdV equation.

Note that all stability works above are based on $L^{2}$ integrability perturbation. Namely the initial perturbation is $L^{2}$ integrable. If the initial perturbation is space  periodic, what about the stability of viscous shock? When $\textcolor{blue}{\mu}{\textcolor{red}{=}}\textcolor{blue}{\gamma}{\textcolor{red}{=}}0$, the KdV-Burgers equation becomes hyperbolic conservation law  and  the periodic perturbation problem is  interesting and challenging since the solution oscillates at the far field and resonance may happen\cite{mr1984}. Indeed, Lax \cite{L1957} and Dafermos \cite{D1995} proved that the solution time asymptotically   tends to the periodic average. The asymptotic stability of shocks in both inviscid and viscid cases with periodic perturbation was obtained in \cite{XYY2019,XYY2021,YY2019}. We refer to  \cite{HY2021,Hy2021}  for more interesting works. When $\textcolor{blue}{\gamma}{\textcolor{red}{=}}0$, the KdV-Burgers equation becomes  KdV equation, and the periodic solutions were studied in \cite{S1997}.
Motivated by \cite{Hy2021}, we wonder whether the  viscous shock wave constructed  in \cite{BS1985,YZZ2009}  for the KdV-Burgers equation \eqref{1.1} is time asymptotically stable with periodic perturbations.

In this paper, we consider a Cauchy problem of \eqref{1.1}  with the initial data    satisfying
\begin{equation}\label{1.2}
u_{0}{\textcolor{blue}{(x)}}:{\textcolor{red}{=}} u\textcolor{blue}{(x,0)}\rightarrow \left\{\begin{array}{l}
\bar{u}_{l}+w_{0l}{\textcolor{blue}{(x)}}, \quad \quad  x\rightarrow   -\infty,\\
\bar{u}_{r}+w_{0r}{\textcolor{blue}{(x)}}, \quad \quad x \rightarrow \infty,
\end{array} \quad\right.
\end{equation}
where $\bar{u}_{l}$,   $\bar{u}_{r}$    are constants satisfying $\bar{u}_{l}>\bar{u}_{r}$. Function $ w_{0 i}{\textcolor{blue}{(x)}} \in L^{\infty}(\mathbb{R})$   is a periodic function with period $p_i>0$,($i{\textcolor{red}{=}}r,l$) satisfying
\begin{equation}\label{1.3}
\frac{1}{p_{i}} \int_{0}^{p_{i}} w_{0i}{\textcolor{blue}{(x)}}dx{\textcolor{red}{=}} 0.
\end{equation}


We aim to prove that the   shock wave is stable for the Cauchy problem \eqref{1.1}-\eqref{1.2}. Roughly speaking, the solution not only exists globally but also tends to a viscous shock wave  as time goes to infinity. Moreover,  the exponential time decay rate toward the viscous shock wave is also obtained for some certain periodic
perturbations.  The precise statements of the main results are given in Theorem \ref{theorem201} and Theorem \ref{theorem2.1} in Section 2.
We outline the strategy as follows. We apply the anti-derivative method to study the stability of the traveling wave solution  $\textcolor{blue}{\phi}$, in which the anti-derivative of the perturbation $u-\textcolor{blue}{\phi}$, namely,  $\Psi\textcolor{blue}{(x,t)}{\textcolor{red}{=}}\int_{-\infty}^{x}(u-\textcolor{blue}{\phi})\textcolor{blue}{(y,t)}dy$, ``should'' belong to some Sobolev spaces like $H^2(\mathbb{R})$. However, the method above can not be applicable directly in this paper since $u-\textcolor{blue}{\phi}$ oscillates at the far field and hence does not belong to any $L^p$ space for $p\ge 1$. Motivated by \cite{XYY2019},  we introduce a suitable ansatz $\textcolor{blue}{U}\textcolor{blue}{(x,t)}$, which has the same oscillations as the solution $u\textcolor{blue}{(x,t)}$ at the far field, so that $\int_{-\infty}^{x}u\textcolor{blue}{(x,t)}-\textcolor{blue}{U}\textcolor{blue}{(x,t)} dx$ belongs to some Sobolev spaces and the anti-derivative method is still available.

The ansatz is defined as $\textcolor{blue}{U}$ ${\textcolor{red}{=}} u_{l}\textcolor{blue}{(x,t)}g_{ \eta } (x -st-\eta(t) ) + u _ { r } \textcolor{blue}{(x,t)} [ 1 - g { } ( x-st-\eta(t) ) ]$, where $s$ is the shock speed, $u_{ l }$ is a periodic solution of \eqref{1.1} with the initial data $\bar{u}_{l}+w_{0 l}{\textcolor{blue}{(x)}}$ in \eqref{1.2} and is expected to have the same oscillation as $u$ near $x{\textcolor{red}{=}} -\infty$. Similarly $u_r$ is expected to have the same oscillation as $u$ near $x{\textcolor{red}{=}}\infty$. Thus $u\textcolor{blue}{(x,t)}-\textcolor{blue}{U}\textcolor{blue}{(x,t)}$ is expected to be integrable.
The shift $\eta(t)$ in the function
$g_\eta$, related to the traveling wave $\textcolor{blue}{\phi}$, is determined through an ODE equation $\eqref{x2.12}$ and partially depends on the initial periodic perturbations $w_{0 l}$ and $w_{0 r}$. The shift $\eta$ is used to guarantee the integral $\int_{-\infty}^{\infty}(u-\textcolor{blue}{U})\textcolor{blue}{(y,t)}dy{\textcolor{red}{=}}0$ so that $\Psi\textcolor{blue}{(\pm \infty,t)}{\textcolor{red}{=}}0$ and thus $\Psi$ could belong to $L^2$. Moreover,  we obtain the   exponential time convergence rate   toward the viscous shock wave by  using a  weighted  energy estimate provided that   the   initial  perturbation  is located in a suitable weighted function space.

The rest of the paper will be arranged as follows. In Section \ref{section2}, a suitable ansatz is constructed and the main results are stated. In Section \ref{section3},  the stability problem is reformulated to a perturbation equation around the ansatz. In Section \ref{section4},   some weighted a priori estimates are established.   In Section \ref{section5},   the main results   are  proved.    In Section \ref{section6}, some complementary proofs     are provided.
\noindent

\textbf { Notation.} The functional $\|\cdot\|_{L^p(\Omega)}$ is defined by $\| f\|_{L^p(\Omega)} {\textcolor{red}{=}}  (\int_{\Omega}|f|^{p}(\xi ){ d }\xi )^{\frac{1}{p}}$. The symbol $\Omega$ is often omitted  when $\Omega{\textcolor{red}{=}}(-\infty,\infty)$. We denote for simplicity
\begin{equation*}
\| f\| {\textcolor{red}{=}}  \left(\int_{ -\infty}^{ \infty}f^{2}(\xi ){ d }\xi \right)^{\frac{1}{2}}
\end{equation*}
as $p{\textcolor{red}{=}}2$. In addition, $H^m$ denotes the  $m$-th  order Sobolev space of functions defined by
 \begin{equation*}
\|f\|_{m} {\textcolor{red}{=}}  \left( \sum_{k{\textcolor{red}{=}}0}^{m}  \|\partial^{k}_{\xi}f\|^2 \right)^{\frac{1}{2}}.
\end{equation*}

 \section{Preliminaries and  Main Results}\label{section2}
\subsection{ Suitable   Ansatz}
A  viscous shock profile $ \textcolor{blue}{\phi}(x-st)$ is a  traveling wave solution of \eqref{1.1}. It satisfies
\begin{equation}\label{x2.2}
 \left\{ \begin{array}{ll}
&\textcolor{blue}{\mu} \textcolor{blue}{\phi} '''- \textcolor{blue}{\gamma} \textcolor{blue}{\phi}''-s\textcolor{blue}{\phi}' + f'(\textcolor{blue}{\phi}) \textcolor{blue}{\phi}' {\textcolor{red}{=}}0, \\
&\lim_{\xi \rightarrow -\infty }\textcolor{blue}{\phi}(\xi){\textcolor{red}{=}}{\bar{u}_{l}},\quad\lim_{\xi \rightarrow +\infty}\textcolor{blue}{\phi}(\xi){\textcolor{red}{=}}{\bar{u}_{r}},     \\
\end{array} \right.
\end{equation}
where $\textcolor{blue}{\xi} {\textcolor{red}{=}}x-st $, $'{\textcolor{red}{=}}\frac{d}{d \textcolor{blue}{\xi}}$, $s$ is the shock speed defined by Rankine--Hugoniot condition
\begin{equation}\label{x2.3}
s(\bar { u } _ { l} - \bar { u } _ { r } ){\textcolor{red}{=}} f(\bar { u } _ { l })    - f(\bar { u } _ { r }).
\end{equation}
\begin{lma}\cite[Lemma 2.1]{YZZ2009}\label{yl1}
Assume that the Rankine--Hugoniot condition   \eqref{x2.3}   and $\textcolor{blue}{\gamma}^{2}+4 \textcolor{blue}{\mu}(s-f'(\bar{u}_{l}))\geq0$ holds, Then the equation \eqref{x2.2}     has a unique solution $\textcolor{blue}{\phi}(\textcolor{blue}{\xi})$, up to a shift. Moreover, it satisfies $\textcolor{blue}{\phi}'(\textcolor{blue}{\xi})< 0$ and $\bar{u}_r <\textcolor{blue}{\phi}(\textcolor{blue}{\xi}) < \bar{u}_l$ for all $\textcolor{blue}{\xi} \in \mathbb{R}$.
\end{lma}

Let $g ( \textcolor{blue}{\xi} ) :{\textcolor{red}{=}} \frac { \textcolor{blue}{\phi} ( \textcolor{blue}{\xi} ) - \bar { u }_{ r } } { \bar { u }_{ l } - \bar { u }_{ r } }$, it holds that

\begin{equation}
\lim_{\textcolor{blue}{\xi} \rightarrow -\infty} g  (\textcolor{blue}{\xi}){\textcolor{red}{=}}1, \lim_{\textcolor{blue}{\xi} \rightarrow +\infty} g  (\textcolor{blue}{\xi}){\textcolor{red}{=}}0.
\end{equation}
Then we have
\begin{lma}\label{yl3}
There exists a positive constant  $   \sigma_0 $ such that
\begin{equation*}
 \left\{ \begin{array}{ll}
  0< g ({\textcolor{blue}{\xi}}) \leqslant C   e^{-\sigma_0 {\textcolor{blue}{\xi}}}, \quad\quad\quad\quad &{\textcolor{blue}{\xi}}>0,\\
 0< 1- g({\textcolor{blue}{\xi}}) \leqslant C  e^{ \sigma_0{\textcolor{blue}{\xi}}},\quad\quad\quad \quad &{\textcolor{blue}{\xi}}<0,\\
  g '({\textcolor{blue}{\xi}})<0                    , \\
   \mid     g ^{(m)}({\textcolor{blue}{\xi}})  \mid                    \leqslant C e^{- \sigma_0 |{\textcolor{blue}{\xi}}|},\quad &{\textcolor{blue}{\xi}} \in \mathbb{R},m \in \mathbb{N^{+}}.
\end{array} \right.
\end{equation*}
 Here  $C >0 $ depends on $ \bar{u}_{l}$ and $\bar{u}_{r} $.
\end{lma}

\begin{proof}
This lemma is a corollary of \cite[Lemma 2.1]{YZZ2009}  and the proof is omitted.
\end{proof}

We assume that the  function  $u_i\textcolor{blue}{(x,t)} $  is   the   solution  of \eqref{1.1} with the  initial data ($ i{\textcolor{red}{=}}l,r$):
\begin{equation*}
 u_{ i 0 } ^ { } {\textcolor{blue}{(x)}}:{\textcolor{red}{=}}\bar { u }_{ i} + w _{ 0 i } {\textcolor{blue}{(x)}}.
\end{equation*}
We introduce  $g_{  {\textcolor{blue}{\eta }} } \textcolor{blue}{(x,t)}$ by
\begin{equation}
 g_{  {\textcolor{blue}{\eta }} } \textcolor{blue}{(x,t)} :{\textcolor{red}{=}} g ( {\textcolor{blue}{\xi}} -  {\textcolor{blue}{\eta }} ( t ) ),\label{x2.5}
\end{equation}
satisfying
\begin{equation}
\lim_{x\rightarrow -\infty} g_{  {\textcolor{blue}{\eta }}  } \textcolor{blue}{(x,t)}{\textcolor{red}{=}}1, \lim_{x\rightarrow +\infty} g_{   {\textcolor{blue}{\eta }} } (x ,t ){\textcolor{red}{=}}0,\label{y2.5}
\end{equation}
where $ {\textcolor{blue}{\eta }}(t)$ is the shift of the shock profile, the exact expression of $ {\textcolor{blue}{\eta }} (t)$    can be found in \eqref{x2.12}. Motivated by \cite{XYY2019}, we construct an ansatz below
\begin{equation}\label{x2.6}
\textcolor{blue}{U}\textcolor{blue}{(x,t)}:{\textcolor{red}{=}} u_{l}\textcolor{blue}{(x,t)}g_{ {\textcolor{blue}{\eta }}}\textcolor{blue}{(x,t)}+u_{r}\textcolor{blue}{(x,t)}[1-g_{ {\textcolor{blue}{\eta }}}\textcolor{blue}{(x,t)}].
\end{equation}

  Note that $u_{i}, i{\textcolor{red}{=}}l,r$ is a periodic solution of \eqref{1.1} with the initial data $\bar{u}_{i}+w_{0 i}{\textcolor{blue}{(x)}}$ and is expected to have the same oscillation as $u$ near $x{\textcolor{red}{=}} \mp\infty$, respectively.   Thus $\textcolor{blue}{U}\textcolor{blue}{(x,t)}$ is expected to have the same oscillation as $u\textcolor{blue}{(x,t)}$ at the far fields.

\subsection{Location of the Shift $ {\textcolor{blue}{\eta }}(t)$}
At the beginning of this subsection, we list a useful lemma, which will be used later. Equation \eqref{1.1} can be  rewritten  in the new coordinate  $\textcolor{blue}{({\textcolor{blue}{\xi}},t)}$ as
\begin{align}\label{s1.1}
u_{t}-s u_{{\textcolor{blue}{\xi}}} + f(u)_{{\textcolor{blue}{\xi}}}+\textcolor{blue}{\mu} u_{{\textcolor{blue}{\xi}}{\textcolor{blue}{\xi}}{\textcolor{blue}{\xi}}}{\textcolor{red}{=}}\textcolor{blue}{\gamma} u_{{\textcolor{blue}{\xi}}{\textcolor{blue}{\xi}}}, & \ {\textcolor{blue}{\xi}}\in \mathbb{R},\ t>0.
\end{align}

\begin{lma}\label{yl2}
Assume that $u_0\in H^{k+1} (0,p)  $  is a periodic function with period $p>0$ for any integer $k\geq0$. Then the periodic solution $u\textcolor{blue}{({\textcolor{blue}{\xi}},t)}$ of \eqref{s1.1} satisfies
\begin{equation*}
\begin{array}{l}
\left\| \partial_{{\textcolor{blue}{\xi}}}^{k}  (u-\bar{u})  \right\|_{L^{\infty}(\mathbb{R})} \leq C  \|u_{0}-\bar{u} \|_{H^{k+1}(0,p)} e^{-\theta t}, \quad t \geqslant 0,
\end{array}
\end{equation*}
where $\bar{u}{\textcolor{red}{=}}\frac{1}{p}\int_{0}^{p}u_0({\textcolor{blue}{\xi}})d{\textcolor{blue}{\xi}} $ and the positive constants $ C$, $\theta$  are independent of time \(t\).
\end{lma}
 The proof of  {Lemma} \ref{yl2}  is left in Appendix.

 \

 Now we begin to  study the property of the shift. Since $ \textcolor{blue}{U} $ is not the solution of the KdV-Burgers  equation \eqref{1.1}, the error term is
\begin{equation}\label{x2.8}
h:{\textcolor{red}{=}}\textcolor{blue}{U}_{t}-s \textcolor{blue}{U}_{{\textcolor{blue}{\xi}}}+f(\textcolor{blue}{U})_{{\textcolor{blue}{\xi}}}+\textcolor{blue}{\mu} \textcolor{blue}{U}_{{\textcolor{blue}{\xi}}{\textcolor{blue}{\xi}}{\textcolor{blue}{\xi}}}-\textcolor{blue}{\gamma} \textcolor{blue}{U}_{{\textcolor{blue}{\xi}}{\textcolor{blue}{\xi}}}.
\end{equation}
By $\eqref{x2.6}$, we have
\begin{align}\label{x2.9}
\begin{split}
h{\textcolor{red}{=}}&\left\{ f(\textcolor{blue}{U})-f(u_{l})g_{ {\textcolor{blue}{\eta }}}
 -f( u_{r})(1-g_{ {\textcolor{blue}{\eta }}})+3\textcolor{blue}{\mu}(u_{l}-u_{r})_{{\textcolor{blue}{\xi}}}g_{ {\textcolor{blue}{\eta }}}'{-2\textcolor{blue}{\gamma}(u_{l}-u_{r}) g_{ {\textcolor{blue}{\eta }}}'}\right\}_{\textcolor{blue}{\xi}}\\
 &+(f(u_{l})-f(u_{r}))g_{ {\textcolor{blue}{\eta }}}'+\textcolor{blue}{\gamma} [u_{l}-u_{r}] g_{ {\textcolor{blue}{\eta }} }''+\textcolor{blue}{\mu} [u_{l}-u_{r}] g _ {  {\textcolor{blue}{\eta }} }'''\\
 &-(u_{l}-u_{r}) \cdot g _ {  {\textcolor{blue}{\eta }} }' \cdot\left(s+ {\textcolor{blue}{\eta }}'(t)\right).
\end{split}&
\end{align}

Subtracting \eqref{x2.8} from \eqref{s1.1} and integrating the resulting system with respect to ${\textcolor{blue}{\xi}}$ over $(-\infty,\infty)$, one has
\begin{equation*}
\frac{d}{dt}\int_{-\infty}^{\infty}(u-\textcolor{blue}{U} )\textcolor{blue}{({\textcolor{blue}{\xi}},t)}d{\textcolor{blue}{\xi}}{\textcolor{red}{=}}-\int_{-\infty}^{\infty}h\textcolor{blue}{({\textcolor{blue}{\xi}},t)}{d}{\textcolor{blue}{\xi}}.
\end{equation*}
To apply the anti-derivative  method which is often used to study the stability of viscous shock, introduced in \cite{O}, we expect
\begin{equation}\label{x2.10}
 \int_{-\infty}^{\infty}(u-\textcolor{blue}{U})\textcolor{blue}{({\textcolor{blue}{\xi}},t)}{d{\textcolor{blue}{\xi}}}{\textcolor{red}{=}}0,
\end{equation}
holds for any $t\ge 0$. Then we compute from \eqref{x2.10} that
\begin{align}\label{x2.11}
\begin{split}
0{\textcolor{red}{=}}&\int_{-\infty} ^{\infty}\left(f(u_{l})-f(u_{r})\right)g_{ {\textcolor{blue}{\eta }}}'
d{\textcolor{blue}{\xi}}
-\int_{-\infty}^{\infty}\left(u_{l}-u_{r}\right) g_{ {\textcolor{blue}{\eta }} }'\left( {\textcolor{blue}{\eta }}'(t)+s\right)d{\textcolor{blue}{\xi}} \\
&+\textcolor{blue}{\gamma}  \int_{-\infty} ^{\infty}  \left(u_{l}-u_{r}\right)g_{ {\textcolor{blue}{\eta }}}'' d{\textcolor{blue}{\xi}}+ \textcolor{blue}{\mu} \int_{-\infty} ^{\infty}\left(u_{l}-u_{r}\right)g_{ {\textcolor{blue}{\eta }}}'''
d{\textcolor{blue}{\xi}}.
\end{split}&
\end{align}
Thus we obtain the following ODE for $ {\textcolor{blue}{\eta }}(t)$,
\begin{equation}\label{x2.12}
 \left\{\begin{array}{ll}
  {\textcolor{blue}{\eta }}'(t){\textcolor{red}{=}}\frac{\textcolor{blue}{\gamma} \int_{-\infty} ^{\infty}\left(u_{l}-u_{r}\right)g_{ {\textcolor{blue}{\eta }}}''
d{\textcolor{blue}{\xi}}+ \textcolor{blue}{\mu} \int_{-\infty} ^{\infty}\left(u_{l}-u_{r}\right)g_{ {\textcolor{blue}{\eta }}}'''
d{\textcolor{blue}{\xi}}+\int_{-\infty}^{\infty}(f(u_{l})-f( u_{r} ))g_{ {\textcolor{blue}{\eta }}}'
d{\textcolor{blue}{\xi}}}{\int_{-\infty}^{\infty}\left(u_{l}-u_{r}\right)g_{ {\textcolor{blue}{\eta }}}'
d{\textcolor{blue}{\xi}}}-s,\\
 {\textcolor{blue}{\eta }}(0){\textcolor{red}{=}} {\textcolor{blue}{\eta }}_0 .
\end{array} \right.
\end{equation}
And the initial data $ {\textcolor{blue}{\eta }}_0$ of the equation \eqref{x2.12} should satisfy $\int_{-\infty}^{\infty}(u-\textcolor{blue}{U} )({\textcolor{blue}{\xi}},0){d {\textcolor{blue}{\xi}} }{\textcolor{red}{=}}0, $ i.e.,

\begin{align}\label{x2.13}
\begin{split}
&\int_{-\infty}^{\infty}(u_{0}-u_{l0})({\textcolor{blue}{\xi}})g{({\textcolor{blue}{\xi}}- {\textcolor{blue}{\eta }}_0)}+(u_{0}-u_{r0} )({\textcolor{blue}{\xi}})[1-g({\textcolor{blue}{\xi}}- {\textcolor{blue}{\eta }}_0)]d{\textcolor{blue}{\xi}}{\textcolor{red}{=}}0.
\end{split}
\end{align}


\begin{lma}\label{yl4}
There  exists a small constant $\delta_{0}$, such that if    ${\textcolor{red}{\delta}}:{\textcolor{red}{=}}\max \|w_{0i} \|_{H^1(\Omega)}<{\textcolor{red}{\delta}}_{0};i{\textcolor{red}{=}}l,r,$
the ODE problem \eqref{x2.12} has a unique   smooth solution $   {\textcolor{blue}{\eta }}(t):\left[ {0},+\infty\right) \rightarrow \mathbb{R}$. Moreover, the shift $ {\textcolor{blue}{\eta }}(t)$   satisfies
\begin{equation}\label{x2.14}
\begin{array}{l}
\left| {\textcolor{blue}{\eta }}'(t)\right|+ \left| { {\textcolor{blue}{\eta }}}(t)- {\textcolor{blue}{\eta }}_{\infty}\right|\leq C{\textcolor{red}{\delta}} e^{-\theta t},\quad t\geq0,
\end{array}
\end{equation}
where $C $ and $\theta $ are positive constants independent of time $t$, and
\begin{align*}
\begin{split}
 {\textcolor{blue}{\eta }}_{\infty} {\textcolor{red}{=}}&\frac{1}{\bar{u}_{l}-\bar{u}_{r}}\left( {\textcolor{blue}{\eta }}_{\infty, 1}+ {\textcolor{blue}{\eta }}_{\infty, 2}\right),\\
 {\textcolor{blue}{\eta }}_{\infty, 1}{\textcolor{red}{=}}&\int_{-\infty}^{0}(u_{0}-\textcolor{blue}{\phi}-w_{0l})({\textcolor{blue}{\xi}}) {d{\textcolor{blue}{\xi}} }+\int_{0}^{+\infty}\left(u_{0}-\textcolor{blue}{\phi}-w_{0r}\right)({\textcolor{blue}{\xi}})d{\textcolor{blue}{\xi}},\\
 {\textcolor{blue}{\eta }}_{\infty, 2}{\textcolor{red}{=}} &\int_{0}^{+\infty} \frac{1}{p_{l}} \int_{0}^{p_{l}}\left[f\left(u_{l} \right)-f\left(\bar{u}_{l}\right)-su_{l}+s\bar{u}_{l}\right]\textcolor{blue}{({\textcolor{blue}{\xi}},t)}d{\textcolor{blue}{\xi}} dt-\frac{1}{p_{l}} \int_{0}^{p_{l}} \int_{0}^{{\textcolor{blue}{\xi}}} w_{0 l}(y)dy d{\textcolor{blue}{\xi}}\\
 &-\int_{0}^{+\infty} \frac{1}{p_{r}} \int_{0}^{p_{r}}\left[f\left(u_{r} \right)-f\left(\bar{u}_{r}\right)-s u_{r} +s \bar{u}_{r}\right]\textcolor{blue}{({\textcolor{blue}{\xi}},t)}d{\textcolor{blue}{\xi}} dt+\frac{1}{p_{r}} \int_{0}^{p_{r}} \int_{0}^{{\textcolor{blue}{\xi}}} w_{0 r}(y)dy d{\textcolor{blue}{\xi}}.
\end{split}&
\end{align*}
\end{lma}
The proof is left in Section \ref{section6}.
\subsection{Main Theorems}
Based on Lemma \ref{yl4}, we know \eqref{x2.10} holds for any $t\geq 0$ provided that the solution of the equation \eqref{1.1} globally exists. Then we can define the anti-derivative  of the perturbation $ {\textcolor{blue}{\psi }}\textcolor{blue}{({\textcolor{blue}{\xi}},t)}:{\textcolor{red}{=}}u-\textcolor{blue}{U}\textcolor{blue}{({\textcolor{blue}{\xi}},t)}$ by
\begin{equation}
 {\textcolor{blue}{\Psi }}\textcolor{blue}{({\textcolor{blue}{\xi}},t)}:{\textcolor{red}{=}}\int_{ -\infty}^{{\textcolor{blue}{\xi}}} {\textcolor{blue}{\psi }}\textcolor{blue}{(y,t)}dy,
\end{equation}
so that ${\textcolor{blue}{\Psi }}$ belongs to some Sobolev space. We assume  that the initial data satisfies
\begin{equation}\label{x2017}
{\textcolor{blue}{\Psi }}_{0}({\textcolor{blue}{\xi}}):{\textcolor{red}{=}}{\textcolor{blue}{\Psi }}\textcolor{blue}{({\textcolor{blue}{\xi}},0)}\in H^{3}(\mathbb{R}).
\end{equation}
The first result is
\begin{thm}\label{theorem201}  If \eqref{x2.2}, \eqref{x2.3} and \eqref{x2017} hold, there exists a   positive constant $\epsilon_{0}$, such that if
\begin{equation}\label{9.1}
\left\|{\textcolor{blue}{\Psi }}_{0}\right\|_{3}^{2}+{\textcolor{red}{\delta}}_{0}\leq\epsilon_{0}^{2},
\end{equation}
then there exists a unique global solution of \eqref{1.1}, \eqref{1.2} satisfying
\begin{equation*}
\begin{aligned}
&\sup_{{\textcolor{blue}{\xi}} \in \mathbb{R}}|  (u-\textcolor{blue}{\phi}_{ {\textcolor{blue}{\eta }}_{\infty}})\textcolor{blue}{({\textcolor{blue}{\xi}},t)}|\rightarrow 0\quad \forall t\rightarrow \infty,
\end{aligned}
\end{equation*}
where $\textcolor{blue}{\phi}_{ {\textcolor{blue}{\eta }}_\infty}{\textcolor{red}{=}}\textcolor{blue}{\phi}({\textcolor{blue}{\xi}}- {\textcolor{blue}{\eta }}_\infty)$.
\end{thm}

\
In order to obtain time decay rate of the solution, we further assume that
\begin{equation}\label{x2.17}
 \exp \left(\frac{\alpha}{2}\left|{\textcolor{blue}{\xi}}-{\textcolor{blue}{\xi}}_{*}\right|\right)\frac{\partial^{i}}{\partial {\textcolor{blue}{\xi}}^{i}} {\textcolor{blue}{\Psi }}_{0}({\textcolor{blue}{\xi}}) \in L^{2}(\mathbb{R})(i{\textcolor{red}{=}}0,1,2,3),
\end{equation}
where ${\textcolor{blue}{\xi}}_{*}{\textcolor{red}{=}}{\textcolor{blue}{\xi}}^{*}+ {\textcolor{blue}{\eta }}_{\infty}$ and $f'(\textcolor{blue}{\phi}\left({\textcolor{blue}{\xi}}^{*}\right)){\textcolor{red}{=}}s, \quad 0<\alpha< \min\{\frac{2\textcolor{blue}{\gamma}}{3 \textcolor{blue}{\mu}},\sigma_{0}\}.$

\

The second result is

\begin{thm}\label{theorem2.1}  If \eqref{x2.2}, \eqref{x2.3}, \eqref{x2017}, \eqref{x2.17} and \eqref{9.1} hold,
then there exists a unique global solution of \eqref{1.1}, \eqref{1.2} satisfying
\begin{equation*}
\begin{aligned}
 &\sup_{{\textcolor{blue}{\xi}} \in \mathbb{R}}|  (u-\textcolor{blue}{\phi}_{ {\textcolor{blue}{\eta }}_{\infty}})\textcolor{blue}{({\textcolor{blue}{\xi}},t)} |  \leq Ce^{-\beta t},  \quad  \forall t>0.
\end{aligned}
\end{equation*}
Positive constants $\beta $   satisfy
\begin{equation}\label{x2.18}
\left\{\begin{array}{l}
0<\beta < \min\{C_{0}\alpha,\theta\}, \\
 ({C_{0}-\frac{\beta}{\alpha}})(3 \textcolor{blue}{\mu}-  \frac{2 \textcolor{blue}{\gamma}}{\alpha}) +(\alpha\textcolor{blue}{\mu} +\textcolor{blue}{\gamma})^2<0 ,
\end{array}\right.
\end{equation}

\color{blue}
Here
\begin{equation}\label{x2.18}
 \begin{array}{l}
 \beta {\textcolor{red}{=}}\min\{C_{0}\alpha,\theta,C_0-\frac{\alpha(\alpha\textcolor{blue}{\mu} +\textcolor{blue}{\gamma})^{2}}{2\textcolor{blue}{\gamma}- 3 \alpha \textcolor{blue}{\mu}}\},
\end{array}
\end{equation}
\color{black}
where $C_{0}:{\textcolor{red}{=}}G\cdot\left\{\frac{B}{\alpha},\min \left\{\left|\textcolor{blue}{\phi}\left({\textcolor{blue}{\xi}}^{*}-1\right)-\textcolor{blue}{\phi}\left({\textcolor{blue}{\xi}}^{*}\right)\right|,\left|\textcolor{blue}{\phi}\left({\textcolor{blue}{\xi}}^{*}+1\right)-\textcolor{blue}{\phi}\left({\textcolor{blue}{\xi}}^{*}\right)\right|\right\}\right\}$, with $G:{\textcolor{red}{=}}$
$\min_{u \in\left[\bar{u}_{r},\bar{u}_{l}\right]}\left|f''(u)\right|$,
$B:{\textcolor{red}{=}}\min_{{\textcolor{blue}{\xi}}\in\left[{\textcolor{blue}{\xi}}^{*}-1,{\textcolor{blue}{\xi}}^{*}+1\right]}\left|\textcolor{blue}{\phi}'({\textcolor{blue}{\xi}})\right|$.
\end{thm}\color{black}
\section{Reformulation of the Problem}\label{section3}
Subtracting \eqref{x2.8} from \eqref{s1.1} and integrating the resulting system with respect to ${\textcolor{blue}{\xi}}$, we have  that
\begin{equation}\label{x3.2}
{\textcolor{blue}{\Psi }}_{t}-s{\textcolor{blue}{\Psi }}_{{\textcolor{blue}{\xi}}}+[f(u)-f(\textcolor{blue}{U})] +\textcolor{blue}{\mu} {\textcolor{blue}{\Psi }}_{{\textcolor{blue}{\xi}} {\textcolor{blue}{\xi}} {\textcolor{blue}{\xi}}}-\textcolor{blue}{\gamma} {\textcolor{blue}{\Psi }}_{{\textcolor{blue}{\xi}} {\textcolor{blue}{\xi}}}{\textcolor{red}{=}}H,
\end{equation}
where $H:{\textcolor{red}{=}}\int_{-\infty}^{{\textcolor{blue}{\xi}}}-h\textcolor{blue}{({\textcolor{blue}{\xi}},t)} d{\textcolor{blue}{\xi}}.$ We show the following decay properties of the error term of $H$.
\begin{lma}\label{yl6}
The error term $H$ satisfies:
\begin{equation}
\left|\frac{\partial^{j}H }{\partial {\textcolor{blue}{\xi}}^{j}}\right|  \leq C {\textcolor{red}{\delta}}   e^{-\theta t} e^{-\sigma_0|{\textcolor{blue}{\xi}}- {\textcolor{blue}{\eta }}(t)|},\quad \quad  j{\textcolor{red}{=}}0,1,2,3.
\end{equation}
\end{lma}
The proof is left in Section \ref{section5}. We will seek the solution in the functional space $X_{\epsilon}(0,T)$ for any $0\leq T< +\infty $,
\begin{equation*}
X_{\epsilon}(0,T):{\textcolor{red}{=}}\left\{  {\textcolor{blue}{\Psi }} \in C ([0,T];H^3)\Big| {\textcolor{blue}{\Psi }}_{{\textcolor{blue}{\xi}}} \in  L^2(0,T;H^3),\sup_{0\leq t\leq T}\| {\textcolor{blue}{\Psi }}\|_3(t)\leq\epsilon  \right\}.
\end{equation*}
\begin{pro} {\label{pp3}} $(a$ priori estimate) Suppose  that $ {\textcolor{blue}{\Psi }}\in X_{\epsilon}(0,T)$ is the solution of \eqref{x3.2}, \eqref{x2017} for some time $T>0$. There exists a positive constant $\epsilon_0$ independent of  $T$, such that if
\begin{equation*}
\sup_{0\leq t\leq T}\|{\textcolor{blue}{\Psi }}(t)\|_{3}\leq \epsilon \leq \epsilon_{0}
\end{equation*}
for $t\in[0,T],$ then
\begin{equation}
\|{\textcolor{blue}{\Psi }}(t)\|_{3}^{2}+\int_{0}^{t}\left\|{\textcolor{blue}{\Psi }}_{{\textcolor{blue}{\xi}}}(\tau)\right\|_{3}^{2}
d \tau \leq  C_{1} \left(\left\|{\textcolor{blue}{\Psi }}_{0}\right\|_{3}^{2} +{\textcolor{red}{\delta}} \right),
\end{equation}
for any $t\geq 0$. Here $C_{1}$ is independent of $T$.
\end{pro}

\section{Weighted Estimates}\label{section4}
Throughout this  section, we assume that the problem \eqref{x3.2}, \eqref{x2017} has a solution ${\textcolor{blue}{\Psi }}\in X_{\epsilon}(0,T) $.
At the begin of this section, we give a lemma and some equalities which will be use later.
\begin{lma}\label{yl5}
\textcolor{blue}{U}nder the same condition of Lemma \ref{yl4}, one gets
\begin{equation}
\left|\frac{\partial^{i}}{\partial{\textcolor{blue}{\xi}}^{i}}(\textcolor{blue}{U} -\textcolor{blue}{\phi}_{ {\textcolor{blue}{\eta }}_{\infty}})\right| \leq C {\textcolor{red}{\delta}} e^{-\theta t}, \quad i{\textcolor{red}{=}}0,1.
\end{equation}
\end{lma}
\begin{proof}
By directly calculate, one gets
\begin{equation*}
\begin{aligned}
&\textcolor{blue}{U} -\textcolor{blue}{\phi}_{ {\textcolor{blue}{\eta }}_{\infty}}\\
{\textcolor{red}{=}}& [\textcolor{blue}{\phi}_{ {\textcolor{blue}{\eta }}}({\textcolor{blue}{\xi}})-\textcolor{blue}{\phi}_{ {\textcolor{blue}{\eta }}_{\infty}}({\textcolor{blue}{\xi}})]+\left\{(u_{l}-\bar{u}_{l})\frac{\textcolor{blue}{\phi}_{ {\textcolor{blue}{\eta }}}({\textcolor{blue}{\xi}})-\bar{u}_{r}}{\bar{u}_{l}-\bar{u}_{r}}-
(u_{r}-\bar{u}_{r})\frac{\textcolor{blue}{\phi}_{ {\textcolor{blue}{\eta }}}({\textcolor{blue}{\xi}})-\bar{u}_{l}}{\bar{u}_{l}-\bar{u}_{r}}\right\}\\
 :{\textcolor{red}{=}}&W_{1}+W_{2}.
\end{aligned}
\end{equation*}
With the aid of Lemma \ref{yl4}, we have
\begin{equation*}
\begin{aligned}
|W_{1}|,|W_{1{\textcolor{blue}{\xi}}}| \leq &C  | {\textcolor{blue}{\eta }} (t)- {\textcolor{blue}{\eta }}_{\infty}|\leq C  {\textcolor{red}{\delta}}  e^{-\theta t }.
\end{aligned}
\end{equation*}

By  Lemma \ref{yl2}, it follows that $|W_{2}| , |W_{2{\textcolor{blue}{\xi}}}|  \leq C  {\textcolor{red}{\delta}}  e^{-\theta t }.$ Thus the proof of Lemma  \ref{yl5} is obtained.
 \end{proof}
Lagrange mean value theorem gives that
\begin{align*}
\begin{split}
 2 {\textcolor{blue}{\Psi }} \left[f\left(\textcolor{blue}{U} +{\textcolor{blue}{\Psi }}_{{\textcolor{blue}{\xi}}}\right)-f(\textcolor{blue}{U} )\right] {\textcolor{red}{=}}&\left\{f'(\textcolor{blue}{U} ) {\textcolor{blue}{\Psi }}^{2}\right\}_{{\textcolor{blue}{\xi}}}-f''(\textcolor{blue}{U} ) \textcolor{blue}{U}_{{\textcolor{blue}{\xi}}}  {\textcolor{blue}{\Psi }}^{2}+f''\left({\textcolor{blue}{\xi}}_{1}\right) {\textcolor{blue}{\Psi }} {\textcolor{blue}{\Psi }}_{{\textcolor{blue}{\xi}}}^{2},\\
 2 {\textcolor{blue}{\Psi }}_{{\textcolor{blue}{\xi}}}\left[f\left(\textcolor{blue}{U}+{\textcolor{blue}{\Psi }}_{{\textcolor{blue}{\xi}}}\right)-f(\textcolor{blue}{U})\right]_{{\textcolor{blue}{\xi}}}{\textcolor{red}{=}}&2 {\textcolor{blue}{\Psi }}_{{\textcolor{blue}{\xi}}}\left[f'\left({\textcolor{blue}{\xi}}_{2}\right) {\textcolor{blue}{\Psi }}_{{\textcolor{blue}{\xi}}}\right]_{{\textcolor{blue}{\xi}}}{\textcolor{red}{=}}\left\{2 f'\left({\textcolor{blue}{\xi}}_{2}\right) {\textcolor{blue}{\Psi }}_{{\textcolor{blue}{\xi}}}^{2}\right\}_{{\textcolor{blue}{\xi}}}-2 f'\left({\textcolor{blue}{\xi}}_{2}\right) {\textcolor{blue}{\Psi }}_{{\textcolor{blue}{\xi}}} {\textcolor{blue}{\Psi }}_{{\textcolor{blue}{\xi}} {\textcolor{blue}{\xi}}},\\
 2 {\textcolor{blue}{\Psi }}_{{\textcolor{blue}{\xi}} {\textcolor{blue}{\xi}}}\left[f\left(\textcolor{blue}{U}+{\textcolor{blue}{\Psi }}_{{\textcolor{blue}{\xi}}}\right)-f(\textcolor{blue}{U})\right]_{{\textcolor{blue}{\xi}} {\textcolor{blue}{\xi}}}
 {\textcolor{red}{=}}& \left\{2 {\textcolor{blue}{\Psi }}_{{\textcolor{blue}{\xi}} {\textcolor{blue}{\xi}}}\left[f\left(\textcolor{blue}{U}+{\textcolor{blue}{\Psi }}_{{\textcolor{blue}{\xi}}}\right)-f(\textcolor{blue}{U})\right]_{{\textcolor{blue}{\xi}}}\right\}_{{\textcolor{blue}{\xi}}}\\
  &-2 {\textcolor{blue}{\Psi }}_{{\textcolor{blue}{\xi}} {\textcolor{blue}{\xi}} {\textcolor{blue}{\xi}}}\left[f''\left({\textcolor{blue}{\xi}}_{3}\right) \textcolor{blue}{U}_{{\textcolor{blue}{\xi}}} {\textcolor{blue}{\Psi }}_{{\textcolor{blue}{\xi}}}+f'\left(\textcolor{blue}{U}+{\textcolor{blue}{\Psi }}_{{\textcolor{blue}{\xi}}}\right) {\textcolor{blue}{\Psi }}_{{\textcolor{blue}{\xi}} {\textcolor{blue}{\xi}}}\right],\\
 2 {\textcolor{blue}{\Psi }}_{{\textcolor{blue}{\xi}} {\textcolor{blue}{\xi}}{\textcolor{blue}{\xi}}}\left[f\left(\textcolor{blue}{U}+{\textcolor{blue}{\Psi }}_{{\textcolor{blue}{\xi}}}\right)-f(\textcolor{blue}{U})\right]_{{\textcolor{blue}{\xi}}{\textcolor{blue}{\xi}} {\textcolor{blue}{\xi}}}
{\textcolor{red}{=}} & \left\{2 {\textcolor{blue}{\Psi }}_{{\textcolor{blue}{\xi}}{\textcolor{blue}{\xi}} {\textcolor{blue}{\xi}}}\left[f\left(\textcolor{blue}{U}+{\textcolor{blue}{\Psi }}_{{\textcolor{blue}{\xi}}}\right)-f(\textcolor{blue}{U})\right]_{{\textcolor{blue}{\xi}} {\textcolor{blue}{\xi}}}\right\}_{{\textcolor{blue}{\xi}}}
 \\&-2 {\textcolor{blue}{\Psi }}_{{\textcolor{blue}{\xi}} {\textcolor{blue}{\xi}} {\textcolor{blue}{\xi}} {\textcolor{blue}{\xi}}}\left[f''\left({\textcolor{blue}{\xi}}_{3}\right) \textcolor{blue}{U}_{{\textcolor{blue}{\xi}}{\textcolor{blue}{\xi}}} {\textcolor{blue}{\Psi }}_{{\textcolor{blue}{\xi}}}+f'\left(\textcolor{blue}{U}+{\textcolor{blue}{\Psi }}_{{\textcolor{blue}{\xi}}}\right) {\textcolor{blue}{\Psi }}_{{\textcolor{blue}{\xi}} {\textcolor{blue}{\xi}}{\textcolor{blue}{\xi}}}+f''\left(\textcolor{blue}{U}+{\textcolor{blue}{\Psi }}_{{\textcolor{blue}{\xi}}}\right) {\textcolor{blue}{\Psi }}^{2}_{{\textcolor{blue}{\xi}}{\textcolor{blue}{\xi}}}  \right]\\
 &-2 {\textcolor{blue}{\Psi }}_{{\textcolor{blue}{\xi}} {\textcolor{blue}{\xi}} {\textcolor{blue}{\xi}} {\textcolor{blue}{\xi}}}[2 f''\left(\textcolor{blue}{U}+{\textcolor{blue}{\Psi }}_{{\textcolor{blue}{\xi}}}\right) {\textcolor{blue}{\Psi }}_{{\textcolor{blue}{\xi}} {\textcolor{blue}{\xi}}} \textcolor{blue}{U}_{{\textcolor{blue}{\xi}}} +f'''\left({\textcolor{blue}{\xi}}_{4}\right) {\textcolor{blue}{\Psi }}_{{\textcolor{blue}{\xi}}} \textcolor{blue}{U}_{{\textcolor{blue}{\xi}}}^{2}],
 \end{split}&
\end{align*}
where  ${\textcolor{blue}{\xi}}_{i}$   between $\textcolor{blue}{U} $ and $\textcolor{blue}{U} +{\textcolor{blue}{\Psi }}_{{\textcolor{blue}{\xi}}}, i{\textcolor{red}{=}}1,2,3,4$.
\begin{lma}\label{yl9}
\textcolor{blue}{U}nder the same assumptions in Proposition \ref{pp3}, if $ \alpha,\beta $ satisfy \eqref{x2.18} or $\alpha{\textcolor{red}{=}}\beta{\textcolor{red}{=}}0$, we have the following inequality
\begin{equation}
\begin{aligned}
&e^{\beta t} \int_{-\infty}^{+\infty} e^{\alpha\left|{\textcolor{blue}{\xi}}-{\textcolor{blue}{\xi}}_{*}\right|} {\textcolor{blue}{\Psi }} ^{2}(t, {\textcolor{blue}{\xi}}) d {\textcolor{blue}{\xi}}+\int_{0}^{t} e^{\beta \tau} \int_{-\infty}^{+\infty} e^{\alpha\left|{\textcolor{blue}{\xi}}-{\textcolor{blue}{\xi}}_{*}\right|} {\textcolor{blue}{\Psi }}_{{\textcolor{blue}{\xi}}}^{2}(\tau, {\textcolor{blue}{\xi}}) d {\textcolor{blue}{\xi}} d \tau \\
\quad \leq &C \int_{-\infty}^{+\infty} e^{\alpha\left|{\textcolor{blue}{\xi}}-{\textcolor{blue}{\xi}}_{*}\right|} {\textcolor{blue}{\Psi }}_{0}^{2}  ({\textcolor{blue}{\xi}}) d {\textcolor{blue}{\xi}}+ C {\textcolor{red}{\delta}}.
\end{aligned}
\end{equation}
\end{lma}
\begin{proof}
We  multiply equation \eqref{x3.2} by $2 e^{\beta t} e^{\alpha|{\textcolor{blue}{\xi}}-{\textcolor{blue}{\xi}}_{*}|} {\textcolor{blue}{\Psi }}({\textcolor{blue}{\xi}}, t)$,  and integrate result with respect to $t$ and ${\textcolor{blue}{\xi}}$ over $ [0,t]\times \mathbb{R} $,  we have

\begin{equation}
\begin{aligned}
& e^{\beta t} \int_{-\infty}^{+\infty} e^{\alpha\left|{\textcolor{blue}{\xi}}-{\textcolor{blue}{\xi}}_{*}\right|} {\textcolor{blue}{\Psi }}^{2}(t, {\textcolor{blue}{\xi}}) d {\textcolor{blue}{\xi}}+2 \textcolor{blue}{\gamma} \int_{0}^{t} e^{\beta \tau} \int_{-\infty}^{+\infty} e^{\alpha\left|{\textcolor{blue}{\xi}}-{\textcolor{blue}{\xi}}_{*}\right|} {\textcolor{blue}{\Psi }}_{{\textcolor{blue}{\xi}}}^{2}(\tau, {\textcolor{blue}{\xi}}) d {\textcolor{blue}{\xi}} d \tau \\
 &  + \int_{0}^{t} e^{\beta \tau} \int_{-\infty}^{+\infty} e^{\alpha\left|{\textcolor{blue}{\xi}}-{\textcolor{blue}{\xi}}_{*}\right|}   A_{\alpha}({\textcolor{blue}{\xi}})         {\textcolor{blue}{\Psi }} ^{2}(\tau, {\textcolor{blue}{\xi}}) d {\textcolor{blue}{\xi}} d \tau\\
 :{\textcolor{red}{=}} & \int_{-\infty}^{+\infty} e^{\alpha\left|{\textcolor{blue}{\xi}}-{\textcolor{blue}{\xi}}_{*}\right|} {\textcolor{blue}{\Psi }}_{0}^{2}({\textcolor{blue}{\xi}}) d{\textcolor{blue}{\xi}}+\sum_{i{\textcolor{red}{=}}1}^8    \int_{0}^{t} e^{\beta \tau} \int_{-\infty}^{+\infty}  e^{\alpha\left|{\textcolor{blue}{\xi}}-{\textcolor{blue}{\xi}}_{*}\right|}  a_i d {\textcolor{blue}{\xi}} d \tau,
 \end{aligned}
\label{s4.2}\end{equation}
where

$$A_{\alpha}({\textcolor{blue}{\xi}}) {\textcolor{red}{=}} -\alpha {sign}\left[\left({\textcolor{blue}{\xi}}-{ {\textcolor{blue}{\eta }}_{\infty}}\right)-\left({\textcolor{blue}{\xi}}_{*}-{ {\textcolor{blue}{\eta }}_{\infty}}\right)\right]\left(f'(\textcolor{blue}{\phi}_{ {\textcolor{blue}{\eta }}_{\infty}})-s\right)+f''(\textcolor{blue}{\phi}_{ {\textcolor{blue}{\eta }}_{\infty}})\left|\textcolor{blue}{\phi}_{ {\textcolor{blue}{\eta }}_{\infty},{\textcolor{blue}{\xi}}} \right|,$$

\begin{equation*}
\begin{aligned}
a_1 {\textcolor{red}{=}}& \beta        {\textcolor{blue}{\Psi }}^{2}     ,& a_2 {\textcolor{red}{=}}&-3 \alpha \textcolor{blue}{\mu}       {sign}\left({\textcolor{blue}{\xi}}-{\textcolor{blue}{\xi}}_{*}\right) {\textcolor{blue}{\Psi }}_{{\textcolor{blue}{\xi}}}^{2}    ,\\
a_3 {\textcolor{red}{=}}&-  2 \alpha^{2} \textcolor{blue}{\mu}      {\textcolor{blue}{\Psi }} {\textcolor{blue}{\Psi }}_{{\textcolor{blue}{\xi}}}    , &  a_4 {\textcolor{red}{=}}&- {2 \alpha \textcolor{blue}{\gamma}      {sign}\left({\textcolor{blue}{\xi}}-{\textcolor{blue}{\xi}}_{*}\right) {\textcolor{blue}{\Psi }} {\textcolor{blue}{\Psi }}_{{\textcolor{blue}{\xi}}}   }, \\
a_5 {\textcolor{red}{=}}&-{        f''\left({\textcolor{blue}{\xi}}_{1}\right) {\textcolor{blue}{\Psi }} {\textcolor{blue}{\Psi }}_{{\textcolor{blue}{\xi}}}^{2}     }  ,& a_6 {\textcolor{red}{=}}&{      2H {\textcolor{blue}{\Psi }}    }  , \\
a_7 {\textcolor{red}{=}}&           \left(-\alpha {sign}\left({\textcolor{blue}{\xi}}-{\textcolor{blue}{\xi}}_{*}\right)\left(f'(\textcolor{blue}{\phi}_{ {\textcolor{blue}{\eta }}_{\infty}})-f'(\textcolor{blue}{U})\right)\right)          {\textcolor{blue}{\Psi }} ^{2} ,& a_8 {\textcolor{red}{=}}&         \left(f''(\textcolor{blue}{\phi}_{ {\textcolor{blue}{\eta }}_{\infty}})\left|\textcolor{blue}{\phi}_{ {\textcolor{blue}{\eta }}_{\infty},{\textcolor{blue}{\xi}}} \right|-f''(\textcolor{blue}{U})\left|\textcolor{blue}{U}_{{\textcolor{blue}{\xi}}}\right|\right)          {\textcolor{blue}{\Psi }} ^{2} .\\
\end{aligned}
\end{equation*}
One gets $A_{\alpha}({\textcolor{blue}{\xi}})  \geq   C_{0} \alpha$   (\cite[Lemma 3.1]{YZZ2009}). We   rewrite $    [ 2 \textcolor{blue}{\gamma} {\textcolor{blue}{\Psi }}_{{\textcolor{blue}{\xi}}}^{2} +  C_{0} \alpha {\textcolor{blue}{\Psi }} ^{2}]    -   \Sigma _{j{\textcolor{red}{=}}1}^{4}a_j$  as $({\textcolor{blue}{\Psi }} {\textcolor{blue}{\Psi }}_{{\textcolor{blue}{\xi}}}) \mathbf{M} ({\textcolor{blue}{\Psi }} {\textcolor{blue}{\Psi }}_{{\textcolor{blue}{\xi}}})^{T}$

where the matrix
\begin{align*}
\mathbf{M}:{\textcolor{red}{=}}
\left(
\begin{array}{cccc}
 C_{0} \alpha -\beta    &\,\,    \textcolor{blue}{\mu} \alpha^2+\textcolor{blue}{\gamma} \alpha {sign}\left({\textcolor{blue}{\xi}}-{\textcolor{blue}{\xi}}_{*}\right) \\
\textcolor{blue}{\mu} \alpha^2+\textcolor{blue}{\gamma} \alpha  {sign}\left({\textcolor{blue}{\xi}}-{\textcolor{blue}{\xi}}_{*}\right) &\,\,2 \textcolor{blue}{\gamma}- 3\alpha \textcolor{blue}{\mu} {sign}\left({\textcolor{blue}{\xi}}-{\textcolor{blue}{\xi}}_{*}\right) \\
\end{array}
\right),
\end{align*}

A directly calculation gives if \eqref{x2.18} holds, the matric $\mathbf{M}$ is  positive.  Thus we can find a positive constant $\sigma_1$ such that

\begin{equation}\label{s4.3}
 \left(
\begin{array}{cccc}
 {{\textcolor{blue}{\Psi }}} \, {{\textcolor{blue}{\Psi }}}_{{\textcolor{blue}{\xi}}} \\
\end{array}
\right )
\mathbf{M}
\left(
\begin{array}{cccc}
 {{\textcolor{blue}{\Psi }}} \\
 {{\textcolor{blue}{\Psi }}}_{{\textcolor{blue}{\xi}}}\\
\end{array}
\right )\left\{\begin{array}{l}
>\sigma_1[{{\textcolor{blue}{\Psi }}}^{2} +{{\textcolor{blue}{\Psi }}}_{{\textcolor{blue}{\xi}}}^{2}]     ,  \ \quad \quad \ \alpha>0 ,\\
{\textcolor{red}{=}} 2 \textcolor{blue}{\gamma} {{\textcolor{blue}{\Psi }}}_{{\textcolor{blue}{\xi}}}^{2}     , \quad \quad \quad \quad \quad  \alpha{\textcolor{red}{=}}0.
\end{array} \quad\right.
\end{equation}

Now we estimate the last four terms $A_i$  on the right-hand side of \eqref{s4.2}, where $ A_{i}{\textcolor{red}{=}}\int_{0}^{t} e^{\beta \tau} \int_{-\infty}^{+\infty} a_i d {\textcolor{blue}{\xi}} d \tau$.  With the aid of Sobolev inequality, one gets
\begin{align}
\begin{split}
   A_5  & \leq C\sup_{\tau\in[0,t]}\|{\textcolor{blue}{\Psi }}(\tau)\|_{L^{\infty}}   \int_{0}^{t} e^{\beta \tau} \int_{-\infty}^{+\infty} e^{\alpha\left|{\textcolor{blue}{\xi}}-{\textcolor{blue}{\xi}}_{*}\right|} {\textcolor{blue}{\Psi }}_{{\textcolor{blue}{\xi}}}^{2}(\tau, {\textcolor{blue}{\xi}}) d {\textcolor{blue}{\xi}} d \tau \\
& \leq C \epsilon \int_{0}^{t} e^{\beta \tau} \int_{-\infty}^{+\infty} e^{\alpha\left|{\textcolor{blue}{\xi}}-{\textcolor{blue}{\xi}}_{*}\right|} {\textcolor{blue}{\Psi }}_{{\textcolor{blue}{\xi}}}^{2}(\tau, {\textcolor{blue}{\xi}}) d {\textcolor{blue}{\xi}} d \tau.
\end{split}&
\end{align}
By \eqref{x2.14}, one gets that $ {\textcolor{blue}{\eta }}$ is bounded. Thus, we can find a sufficient big positive constant $N$, such that $N\geq \max\{{\textcolor{blue}{\xi}}_* ,\pm {\textcolor{blue}{\eta }} (t)\}$, if $\beta<\theta, \alpha<\sigma_{0}$, we have
\begin{align}\label{s4.6}
\begin{split}
   A_6 \leq& C\sup_{\tau\in[0,t]} \|{\textcolor{blue}{\Psi }}(\tau)\|_{L^{\infty}}   \int_{0}^{t} e^{\beta \tau} \int_{-\infty}^{+\infty} e^{\alpha\left|{\textcolor{blue}{\xi}}-{\textcolor{blue}{\xi}}_{*}\right|} |H|  d {\textcolor{blue}{\xi}} d \tau \\
\leq& C \epsilon {\textcolor{red}{\delta}}\int_{0}^{t} e^{\beta \tau} e^{-\theta \tau} \int_{-\infty}^{+\infty} e^{\alpha\left|{\textcolor{blue}{\xi}}-{\textcolor{blue}{\xi}}_{*}\right|} e^{- \sigma_0|{\textcolor{blue}{\xi}}- {\textcolor{blue}{\eta }}(t)|}  (\tau, {\textcolor{blue}{\xi}}) d {\textcolor{blue}{\xi}} d \tau\\
{\textcolor{red}{=}} & C \epsilon {\textcolor{red}{\delta}} \int_{0}^{t}e^{(\beta-\theta) \tau}\left(  \int_{-\infty} ^{-N }+\int_{-N} ^{ N}+\int_{N} ^{ \infty}\right)e^{\alpha\left|{\textcolor{blue}{\xi}}-{\textcolor{blue}{\xi}}_{*}\right|}   e^{- \sigma_0 |{\textcolor{blue}{\xi}}- {\textcolor{blue}{\eta }}(t)|}d {\textcolor{blue}{\xi}} d\tau  \\
\leq & C\epsilon \cdot e^{(\sigma_{0}+\alpha)  N }   \cdot {\textcolor{red}{\delta}} \int_{0}^{t}e^{(\beta-\sigma_{0}) \tau}  \int_{-\infty}^{-N}     e^{(\sigma_{0}-\alpha)  {\textcolor{blue}{\xi}}}d {\textcolor{blue}{\xi}} d\tau \\
&+ C \epsilon {\textcolor{red}{\delta}} \int_{0}^{t}e^{(\beta-\theta) \tau}  \int_{-N} ^{ N} e^{\alpha\left|{\textcolor{blue}{\xi}}-{\textcolor{blue}{\xi}}_{*}\right|}   e^{-\sigma_0 |{\textcolor{blue}{\xi}}- {\textcolor{blue}{\eta }}(t)|} d{\textcolor{blue}{\xi}} d\tau\\
&+C\epsilon \cdot e^{({\sigma_0}+\alpha)  N }   \cdot {\textcolor{red}{\delta}} \int_{0}^{t}e^{(\beta-\theta) \tau}\int_{N} ^{\infty}e^{(\alpha-{\sigma_0}){\textcolor{blue}{\xi}}}d {\textcolor{blue}{\xi}} d\tau \leq C{\textcolor{red}{\delta}}.
\end{split}&
\end{align}
With the aid of Lemma \ref{yl5}, if $\alpha\neq 0,$ we have
\begin{align}\label{s4.7}
\begin{split}
  A_7 &\leq \alpha \int_{0}^{t} e^{\beta \tau} \int_{-\infty}^{+\infty} e^{\alpha\left|{\textcolor{blue}{\xi}}-{\textcolor{blue}{\xi}}_{*}\right|}    \left|f'(\textcolor{blue}{\phi}_{ {\textcolor{blue}{\eta }}_{\infty}})-f'(\textcolor{blue}{U})\right|{\textcolor{blue}{\Psi }}^{2}(\tau,{\textcolor{blue}{\xi}}) d {\textcolor{blue}{\xi}} d \tau\\
 & \leq C \int_{0}^{t} e^{\beta\tau} \int_{-\infty}^{+\infty} e^{\alpha\left|{\textcolor{blue}{\xi}}-{\textcolor{blue}{\xi}}_{*}\right|}\left| \textcolor{blue}{\phi}_{ {\textcolor{blue}{\eta }}_{\infty}}- \textcolor{blue}{U}\right|{\textcolor{blue}{\Psi }}^{2}(\tau, {\textcolor{blue}{\xi}})d{\textcolor{blue}{\xi}} d\tau\\
 & \leq C {\textcolor{red}{\delta}} \int_{0}^{t} e^{\beta\tau} \int_{-\infty}^{+\infty} e^{\alpha\left|{\textcolor{blue}{\xi}}-{\textcolor{blue}{\xi}}_{*}\right|} e^{-\theta \tau}
     {\textcolor{blue}{\Psi }}^{2}(\tau,{\textcolor{blue}{\xi}}) d{\textcolor{blue}{\xi}} d\tau\\
 & \leq C {\textcolor{red}{\delta}} \int_{0}^{t} e^{\beta\tau} \int_{-\infty}^{+\infty} e^{\alpha\left|{\textcolor{blue}{\xi}}-{\textcolor{blue}{\xi}}_{*}\right|}{\textcolor{blue}{\Psi }}^{2}(\tau,{\textcolor{blue}{\xi}}) d{\textcolor{blue}{\xi}} d\tau,
 \end{split}&
\end{align}
and
\begin{align}\label{s4.8}
\begin{split}
A_8{\textcolor{red}{=}}&\int_{0}^{t} e^{\beta\tau}\int_{-\infty}^{+\infty} e^{\alpha\left|{\textcolor{blue}{\xi}}-{\textcolor{blue}{\xi}}_{*}\right|}\left(f''(\textcolor{blue}{\phi}_{ {\textcolor{blue}{\eta }}_{\infty}})\left|\textcolor{blue}{\phi}_{ {\textcolor{blue}{\eta }}_{\infty},{\textcolor{blue}{\xi}}} \right|-f''(\textcolor{blue}{\phi}_{ {\textcolor{blue}{\eta }}_{\infty}})\left|\textcolor{blue}{U}_{{\textcolor{blue}{\xi}}}\right|\right){\textcolor{blue}{\Psi }}^{2}(\tau,{\textcolor{blue}{\xi}}) d{\textcolor{blue}{\xi}} d\tau\\
&+\int_{0}^{t}e^{\beta\tau}\int_{-\infty}^{+\infty} e^{\alpha\left|{\textcolor{blue}{\xi}}-{\textcolor{blue}{\xi}}_{*}\right|}\left(f''(\textcolor{blue}{\phi}_{ {\textcolor{blue}{\eta }}_{\infty}})\left|\textcolor{blue}{U}_{ {\textcolor{blue}{\xi}}}\right|-f''(\textcolor{blue}{U})\left|\textcolor{blue}{U}_{{\textcolor{blue}{\xi}}}\right|\right){\textcolor{blue}{\Psi }}^{2}(\tau,{\textcolor{blue}{\xi}}) d{\textcolor{blue}{\xi}} d\tau\\
\leq &C \int_{0}^{t} e^{\beta\tau} \int_{-\infty}^{+\infty} e^{\alpha\left|{\textcolor{blue}{\xi}}-{\textcolor{blue}{\xi}}_{*}\right|}\left|\textcolor{blue}{\phi}_{ {\textcolor{blue}{\eta }}_{\infty},{\textcolor{blue}{\xi}}}-\textcolor{blue}{U}_{{\textcolor{blue}{\xi}}}\right| {\textcolor{blue}{\Psi }}^{2}(\tau,{\textcolor{blue}{\xi}}) d{\textcolor{blue}{\xi}} d\tau\\
&+C\int_{0}^{t}e^{\beta\tau}\int_{-\infty}^{+\infty}e^{\alpha\left|{\textcolor{blue}{\xi}}-{\textcolor{blue}{\xi}}_{*}\right|}  |\textcolor{blue}{U}_{{\textcolor{blue}{\xi}}} ||\textcolor{blue}{\phi}_{ {\textcolor{blue}{\eta }}_{\infty}}-\textcolor{blue}{U}|{\textcolor{blue}{\Psi }}^{2}(\tau,{\textcolor{blue}{\xi}}) d{\textcolor{blue}{\xi}} d\tau\\
\leq&C{\textcolor{red}{\delta}}\int_{0}^{t} e^{\beta\tau} \int_{-\infty}^{+\infty} e^{\alpha\left|{\textcolor{blue}{\xi}}-{\textcolor{blue}{\xi}}_{*}\right|}{\textcolor{blue}{\Psi }}^{2}(\tau,{\textcolor{blue}{\xi}}) d{\textcolor{blue}{\xi}} d\tau.
\end{split}&
\end{align}
On the other hand, for $\alpha{\textcolor{red}{=}}\beta{\textcolor{red}{=}}0$, we have
\begin{align}\label{s4.9}
\begin{split}
A_7{\textcolor{red}{=}}&0,\\
A_8\leq&C\int_{0}^{t}\max\{\left\|\textcolor{blue}{\phi}_{ {\textcolor{blue}{\eta }}_{\infty}}-\textcolor{blue}{U} \right\|_{L^{\infty}},\left\|\textcolor{blue}{\phi}_{ {\textcolor{blue}{\eta }}_{\infty},{\textcolor{blue}{\xi}}}-\textcolor{blue}{U}_{{\textcolor{blue}{\xi}}}\right\|_{L^{\infty}} \}\|{\textcolor{blue}{\Psi }} \|^{2}(\tau,{\textcolor{blue}{\xi}})d\tau\\
\leq &C\sup_{\tau\in[0,t]} \|{\textcolor{blue}{\Psi }}\|^{2}\int_{0}^{t}{\textcolor{red}{\delta}} e^{-\theta \tau}d\tau \leq C  {\textcolor{red}{\delta}},
\end{split}&
\end{align}
where we have used Lemma \ref{yl5}. Substituting \eqref{s4.3}-\eqref{s4.8} into $\eqref{s4.2}$, for $\alpha\neq0$, one has
\begin{align}\label{s4.10}
\begin{split}
&e^{\beta t} \int_{-\infty}^{+\infty} e^{\alpha\left|{\textcolor{blue}{\xi}}-{\textcolor{blue}{\xi}}_{*}\right|} {\textcolor{blue}{\Psi }}^{2}(t,{\textcolor{blue}{\xi}}) d {\textcolor{blue}{\xi}}+\left(\sigma_1 -C{\textcolor{red}{\delta}} \right)\int_{0}^{t} e^{\beta \tau} \int_{-\infty}^{+\infty} e^{\alpha |{\textcolor{blue}{\xi}}-{\textcolor{blue}{\xi}}_{*}|} {\textcolor{blue}{\Psi }}^{2}(\tau,{\textcolor{blue}{\xi}}) d{\textcolor{blue}{\xi}} d\tau \\
 &+\left(\sigma_1-C\epsilon \right) \int_{0}^{t} e^{\beta \tau} \int_{-\infty}^{+\infty} e^{\alpha\left| {\textcolor{blue}{\xi}}-{\textcolor{blue}{\xi}}_{*} \right|} {\textcolor{blue}{\Psi }}_{{\textcolor{blue}{\xi}}}^{2}(\tau,{\textcolor{blue}{\xi}}) d{\textcolor{blue}{\xi}} d\tau \\
\leq& \int_{-\infty}^{+\infty} e^{a\left|{\textcolor{blue}{\xi}}-{\textcolor{blue}{\xi}}_{*}\right|} {\textcolor{blue}{\Psi }}_{0}^{2}({\textcolor{blue}{\xi}}) d {\textcolor{blue}{\xi}}+   C   {\textcolor{red}{\delta}}.
\end{split}
\end{align}
Substituting $\eqref{s4.3}$-$\eqref{s4.7}$ and \eqref{s4.9} into $\eqref{s4.2}$, for $\alpha{\textcolor{red}{=}}\beta{\textcolor{red}{=}}0$, one has
\begin{align}\label{s4.11}
\begin{split}
& \int_{-\infty}^{+\infty}{\textcolor{blue}{\Psi }}^{2}(t,{\textcolor{blue}{\xi}}) d{\textcolor{blue}{\xi}}
 +\left(2\textcolor{blue}{\gamma}-C\epsilon \right)\int_{0}^{t}\int_{-\infty}^{+\infty} {\textcolor{blue}{\Psi }}_{{\textcolor{blue}{\xi}}}^{2}(\tau,{\textcolor{blue}{\xi}}) d{\textcolor{blue}{\xi}} d\tau\\
\leq& \int_{-\infty}^{+\infty}{\textcolor{blue}{\Psi }}_{0}^{2}({\textcolor{blue}{\xi}}) d{\textcolor{blue}{\xi}}+ C{\textcolor{red}{\delta}}.
\end{split}
\end{align}
Combining \eqref{s4.10}, \eqref{s4.11}, if $\alpha,\beta$ satisfy \eqref{x2.18}, we have complete the proof of Lemma \ref{yl9}.
\end{proof}
\begin{lma}\label{yl10}
\textcolor{blue}{U}nder the same assumptions in Lemma \ref{yl9}, we have the following inequality
\begin{equation}\label{s4.12}
\begin{aligned}
&e^{\beta t} \int_{-\infty}^{+\infty} e^{\alpha\left|{\textcolor{blue}{\xi}}-{\textcolor{blue}{\xi}}_{*}\right|} {\textcolor{blue}{\Psi }}_{{\textcolor{blue}{\xi}} }^{2}(t,{\textcolor{blue}{\xi}}) d{\textcolor{blue}{\xi}}+\int_{0}^{t} e^{\beta\tau} \int_{-\infty}^{+\infty} e^{\alpha\left|{\textcolor{blue}{\xi}}-{\textcolor{blue}{\xi}}_{*}\right|} {\textcolor{blue}{\Psi }}_{{\textcolor{blue}{\xi}}{\textcolor{blue}{\xi}}}^{2}(\tau,{\textcolor{blue}{\xi}})d{\textcolor{blue}{\xi}} d\tau \\
\quad \leq &C \int_{-\infty}^{+\infty} e^{\alpha\left|{\textcolor{blue}{\xi}}-{\textcolor{blue}{\xi}}_{*}\right|}\left({\textcolor{blue}{\Psi }}_{0}^{2} +{\textcolor{blue}{\Psi }}_{0,{\textcolor{blue}{\xi}}}^{2} \right)({\textcolor{blue}{\xi}}) d{\textcolor{blue}{\xi}}+ C {\textcolor{red}{\delta}}.
\end{aligned}
\end{equation}
\end{lma}
\begin{proof}
Differentiating \eqref{x3.2} with respect to ${\textcolor{blue}{\xi}}$ once, multiplying the result by $2 e^{\beta t} e^{\alpha|{\textcolor{blue}{\xi}}-{\textcolor{blue}{\xi}}_{*}|} {\textcolor{blue}{\Psi }}_{{\textcolor{blue}{\xi}}}({\textcolor{blue}{\xi}}, t)$, integrating the resulting equation with respect to $t, {\textcolor{blue}{\xi}} $ over $[0, t]  \times \mathbb{R}$, one has
\begin{equation}\label{s4.13}
\begin{aligned}
&e^{\beta t}\int_{-\infty}^{+\infty} e^{\alpha\left|{\textcolor{blue}{\xi}}-{\textcolor{blue}{\xi}}_{*}\right|} {\textcolor{blue}{\Psi }}_{{\textcolor{blue}{\xi}}}^{2}(t,{\textcolor{blue}{\xi}}) d{\textcolor{blue}{\xi}}+2 \textcolor{blue}{\gamma} \int_{0}^{t} e^{\beta \tau} \int_{-\infty}^{+\infty} e^{\alpha\left|{\textcolor{blue}{\xi}}-{\textcolor{blue}{\xi}}_{*}\right|} {\textcolor{blue}{\Psi }}_{{\textcolor{blue}{\xi}}{\textcolor{blue}{\xi}}}^{2}(\tau,{\textcolor{blue}{\xi}}) d{\textcolor{blue}{\xi}} d\tau \\
\leq& \int_{-\infty}^{+\infty} e^{\alpha\left|{\textcolor{blue}{\xi}}-{\textcolor{blue}{\xi}}_{*}\right|} {\textcolor{blue}{\Psi }}_{0{\textcolor{blue}{\xi}}}^{2}({\textcolor{blue}{\xi}}) d{\textcolor{blue}{\xi}}+\sum_{i{\textcolor{red}{=}}1}^4 \int_{0}^{t} e^{\beta \tau} \int_{-\infty}^{+\infty}  e^{\alpha\left|{\textcolor{blue}{\xi}}-{\textcolor{blue}{\xi}}_{*}\right|} b_i d{\textcolor{blue}{\xi}} d\tau    ,
\end{aligned}
\end{equation}
where
\begin{equation}
\begin{aligned}
b_1{\textcolor{red}{=}}&(\beta+\alpha\left|2 f'\left({\textcolor{blue}{\xi}}_{2}\right)-s\right|){\textcolor{blue}{\Psi }}_{{\textcolor{blue}{\xi}}}^{2} ,&b_2{\textcolor{red}{=}}&3 \alpha\textcolor{blue}{\mu} {\textcolor{blue}{\Psi }}_{{\textcolor{blue}{\xi}} {\textcolor{blue}{\xi}}}^{2},\\
b_3{\textcolor{red}{=}}&2(\alpha^{2}\textcolor{blue}{\mu}+\alpha\textcolor{blue}{\gamma}+\left|f'\left({\textcolor{blue}{\xi}}_{2}\right)\right|)   \left|{\textcolor{blue}{\Psi }}_{{\textcolor{blue}{\xi}}} {\textcolor{blue}{\Psi }}_{{\textcolor{blue}{\xi}}{\textcolor{blue}{\xi}}}\right|,&b_4{\textcolor{red}{=}}&2\left|{\textcolor{blue}{\Psi }}_{{\textcolor{blue}{\xi}}{\textcolor{blue}{\xi}}} H \right|.
\end{aligned}
\end{equation}
Now we estimate the last two terms $B_{i}(i{\textcolor{red}{=}}3,4)$ on the right-hand side of $\eqref{s4.13}$, where $B_{i}{\textcolor{red}{=}}\int_{0}^{t} e^{\beta\tau}\int_{-\infty}^{+\infty}  e^{\alpha\left|{\textcolor{blue}{\xi}}-{\textcolor{blue}{\xi}}_{*}\right|} b_i d{\textcolor{blue}{\xi}} d\tau$. With the aid of Cauchy inequality, one gets
\begin{equation}\label{s4.14}
\begin{aligned}
B_3\leq&\varepsilon_{1}\int_{0}^{t}e^{\beta\tau}\int_{-\infty}^{+\infty} e^{\alpha\left|{\textcolor{blue}{\xi}}-{\textcolor{blue}{\xi}}_{*}\right|}{\textcolor{blue}{\Psi }}_{{\textcolor{blue}{\xi}}{\textcolor{blue}{\xi}}}^{2}(\tau,{\textcolor{blue}{\xi}}) d{\textcolor{blue}{\xi}} d\tau\\
&+C_{\varepsilon_{1}} \int_{0}^{t} e^{\beta\tau}\int_{-\infty}^{+\infty} e^{\alpha\left|{\textcolor{blue}{\xi}}-{\textcolor{blue}{\xi}}_{*}\right|} {\textcolor{blue}{\Psi }}_{{\textcolor{blue}{\xi}}}^{2}(\tau,{\textcolor{blue}{\xi}}) d{\textcolor{blue}{\xi}} d\tau.
\end{aligned}
\end{equation}
\textcolor{blue}{U}sing Lemma \ref{yl6}, similar to \eqref{s4.6}, we have
\begin{equation}\label{s4.16}
\begin{aligned}
B_4 \leq&\varepsilon_{1}\int_{0}^{t} e^{\beta\tau} \int_{-\infty}^{+\infty} e^{\alpha\left|{\textcolor{blue}{\xi}}-{\textcolor{blue}{\xi}}_{*}\right|}{\textcolor{blue}{\Psi }}_{{\textcolor{blue}{\xi}}{\textcolor{blue}{\xi}}}^{2}(\tau,{\textcolor{blue}{\xi}}) d{\textcolor{blue}{\xi}} d\tau\\
&+C_{\varepsilon_{1}} \int_{0}^{t} e^{\beta\tau} \int_{-\infty}^{+\infty} e^{\alpha\left|{\textcolor{blue}{\xi}}-{\textcolor{blue}{\xi}}_{*}\right|} H^{2}(\tau,{\textcolor{blue}{\xi}}) d{\textcolor{blue}{\xi}} d\tau\\
\leq&\varepsilon_{1} \int_{0}^{t} e^{\beta\tau}\int_{-\infty}^{+\infty} e^{\alpha\left|{\textcolor{blue}{\xi}}-{\textcolor{blue}{\xi}}_{*}\right|}{\textcolor{blue}{\Psi }}_{{\textcolor{blue}{\xi}}{\textcolor{blue}{\xi}}}^{2}(\tau,{\textcolor{blue}{\xi}}) d{\textcolor{blue}{\xi}} d\tau+C_{\varepsilon_{1}} {\textcolor{red}{\delta}}.
\end{aligned}
\end{equation}
Substituting \eqref{s4.14}-\eqref{s4.16} into \eqref{s4.11}, choosing a sufficiently small constant $\varepsilon_{1}$, with the aid of Lemma \ref{yl9}, we obtain Lemma \ref{yl10}.
\end{proof}
\begin{lma}\label{yl11}
\textcolor{blue}{U}nder the same assumptions in Lemma \ref{yl9}, we have the following inequality
\begin{equation}\label{s4.18}
\begin{aligned}
&e^{\beta t} \int_{-\infty}^{+\infty} e^{\alpha\left|{\textcolor{blue}{\xi}}-{\textcolor{blue}{\xi}}_{*}\right|} {\textcolor{blue}{\Psi }}_{{\textcolor{blue}{\xi}} {\textcolor{blue}{\xi}}}^{2}(t,{\textcolor{blue}{\xi}}) d{\textcolor{blue}{\xi}}+\int_{0}^{t} e^{\beta\tau} \int_{-\infty}^{+\infty} e^{\alpha\left|{\textcolor{blue}{\xi}}-{\textcolor{blue}{\xi}}_{*}\right|} {\textcolor{blue}{\Psi }}_{{\textcolor{blue}{\xi}} {\textcolor{blue}{\xi}} {\textcolor{blue}{\xi}}}^{2}(\tau,{\textcolor{blue}{\xi}}) d{\textcolor{blue}{\xi}} d\tau \\
 \leq & C \int_{-\infty}^{+\infty} e^{\alpha\left|{\textcolor{blue}{\xi}}-{\textcolor{blue}{\xi}}_{*}\right|}\left({\textcolor{blue}{\Psi }}_{0}^{2}+{\textcolor{blue}{\Psi }}_{0 {\textcolor{blue}{\xi}}}^{2}+{\textcolor{blue}{\Psi }}_{0 {\textcolor{blue}{\xi}} {\textcolor{blue}{\xi}}}^{2}\right)({\textcolor{blue}{\xi}}) d{\textcolor{blue}{\xi}}+ C {\textcolor{red}{\delta}}.
\end{aligned}
\end{equation}
\end{lma}
Here $C$ is a positive constant.
\begin{proof}
Differentiating   of \eqref{x3.2} with respect to ${\textcolor{blue}{\xi}}$ twice, multiplying the result by $2 e^{\beta t} e^{\alpha\left|{\textcolor{blue}{\xi}}-{\textcolor{blue}{\xi}}_{*}\right|} {\textcolor{blue}{\Psi }}_{{\textcolor{blue}{\xi}} {\textcolor{blue}{\xi}}}$ and integrating the result with respect to $t,{\textcolor{blue}{\xi}}$ over $[0,t] \times \mathbb{R}$, one gets that
\begin{equation}\label{s4.19}
\begin{aligned}
 &e^{\beta t} \int_{-\infty}^{+\infty} e^{\alpha\left|{\textcolor{blue}{\xi}}-{\textcolor{blue}{\xi}}_{*}\right|} {\textcolor{blue}{\Psi }}_{{\textcolor{blue}{\xi}} {\textcolor{blue}{\xi}}}^{2}(t,{\textcolor{blue}{\xi}}) d{\textcolor{blue}{\xi}}+2 \textcolor{blue}{\gamma} \int_{0}^{t} e^{\beta\tau} \int_{-\infty}^{+\infty} e^{\alpha\left|{\textcolor{blue}{\xi}}-{\textcolor{blue}{\xi}}_{*}\right|} {\textcolor{blue}{\Psi }}_{{\textcolor{blue}{\xi}} {\textcolor{blue}{\xi}} {\textcolor{blue}{\xi}}}^{2}(\tau,{\textcolor{blue}{\xi}}) d{\textcolor{blue}{\xi}} d\tau \\
\leq & \int_{-\infty}^{+\infty} e^{\alpha\left|{\textcolor{blue}{\xi}}-{\textcolor{blue}{\xi}}_{*}\right|} {\textcolor{blue}{\Psi }}_{0 {\textcolor{blue}{\xi}} {\textcolor{blue}{\xi}}}^{2}({\textcolor{blue}{\xi}}) d{\textcolor{blue}{\xi}}+\sum_{i{\textcolor{red}{=}}1}^6 \int_{0}^{t} e^{\beta\tau} \int_{-\infty}^{+\infty} e^{\alpha\left|{\textcolor{blue}{\xi}}-{\textcolor{blue}{\xi}}_{*}\right|} m_i(\tau,{\textcolor{blue}{\xi}}) d{\textcolor{blue}{\xi}} d\tau,
\end{aligned}
\end{equation}
where
\begin{equation*}
\begin{aligned}
m_1{\textcolor{red}{=}}&(\beta+\alpha \left|2 f'\left(\textcolor{blue}{U}+{\textcolor{blue}{\Psi }}_{{\textcolor{blue}{\xi}}}\right)-s\right| ) e^{\alpha\left|{\textcolor{blue}{\xi}}-{\textcolor{blue}{\xi}}_{*}|\right.}{\textcolor{blue}{\Psi }}_{{\textcolor{blue}{\xi}} {\textcolor{blue}{\xi}}}^{2},& m_2{\textcolor{red}{=}}&3 \alpha\textcolor{blue}{\mu}{\textcolor{blue}{\Psi }}_{{\textcolor{blue}{\xi}}{\textcolor{blue}{\xi}} {\textcolor{blue}{\xi}}}^{2},\\
m_3{\textcolor{red}{=}}&{2\alpha\left|f''\left({\textcolor{blue}{\xi}}_{3}\right) \textcolor{blue}{U}_{{\textcolor{blue}{\xi}}} {\textcolor{blue}{\Psi }}_{{\textcolor{blue}{\xi}}} {\textcolor{blue}{\Psi }}_{{\textcolor{blue}{\xi}} {\textcolor{blue}{\xi}}}\right|},&
m_4{\textcolor{red}{=}}&{2\alpha(\textcolor{blue}{\gamma}+\alpha\textcolor{blue}{\mu})\left|{\textcolor{blue}{\Psi }}_{{\textcolor{blue}{\xi}}{\textcolor{blue}{\xi}}} {\textcolor{blue}{\Psi }}_{{\textcolor{blue}{\xi}}{\textcolor{blue}{\xi}}{\textcolor{blue}{\xi}}}\right|},\\
m_5{\textcolor{red}{=}}&  {2\left\{\left|f''\left({\textcolor{blue}{\xi}}_{3}\right) \textcolor{blue}{U}_{{\textcolor{blue}{\xi}}} {\textcolor{blue}{\Psi }}_{{\textcolor{blue}{\xi}}}\right|+\left|f'\left(\textcolor{blue}{U}+{\textcolor{blue}{\Psi }}_{{\textcolor{blue}{\xi}}}\right) {\textcolor{blue}{\Psi }}_{{\textcolor{blue}{\xi}}{\textcolor{blue}{\xi}}}\right| \right\}\left|{\textcolor{blue}{\Psi }}_{{\textcolor{blue}{\xi}}{\textcolor{blue}{\xi}}{\textcolor{blue}{\xi}}}\right|},&
m_6{\textcolor{red}{=}}&{2\left|{\textcolor{blue}{\Psi }}_{{\textcolor{blue}{\xi}}{\textcolor{blue}{\xi}}} H_{{\textcolor{blue}{\xi}}{\textcolor{blue}{\xi}}}\right|}.
\end{aligned}
\end{equation*}
Now we estimate the last four terms $M_{i}(i{\textcolor{red}{=}}3,4,5,6)$ on the right-hand side of $\eqref{s4.19}$, where $ M_{i}{\textcolor{red}{=}}\int_{0}^{t} e^{\beta\tau} \int_{-\infty}^{+\infty}  e^{\alpha\left|{\textcolor{blue}{\xi}}-{\textcolor{blue}{\xi}}_{*}\right|} m_i d{\textcolor{blue}{\xi}} d\tau$. With the aid of Cauchy inequality, we have
\begin{equation}\label{s4.20}
\begin{aligned}
M_{i} \leq & \varepsilon_{2} \int_{0}^{t} e^{\beta\tau} \int_{-\infty}^{+\infty} e^{\alpha\left|{\textcolor{blue}{\xi}}-{\textcolor{blue}{\xi}}_{*}\right|} {\textcolor{blue}{\Psi }}_{{\textcolor{blue}{\xi}}{\textcolor{blue}{\xi}}{\textcolor{blue}{\xi}}}^{2}(\tau,{\textcolor{blue}{\xi}}) d{\textcolor{blue}{\xi}} d\tau\\
&+C_{ \varepsilon_{2}} \int_{0}^{t} e^{\beta\tau} \int_{-\infty}^{+\infty} e^{\alpha\left|{\textcolor{blue}{\xi}}-{\textcolor{blue}{\xi}}_{*}\right|}\left({\textcolor{blue}{\Psi }}_{{\textcolor{blue}{\xi}}}^{2}+{\textcolor{blue}{\Psi }}_{{\textcolor{blue}{\xi}} {\textcolor{blue}{\xi}}}^{2}\right)(\tau,{\textcolor{blue}{\xi}}) d{\textcolor{blue}{\xi}} d\tau,\quad  i{\textcolor{red}{=}}3,4,5.
\end{aligned}
\end{equation}
Similar to \eqref{s4.6}, we have
\begin{equation}\label{s4.23}
\begin{aligned}
M_{6}\leq& C \int_{0}^{t} e^{\beta\tau} \int_{-\infty}^{+\infty} e^{\alpha\left|{\textcolor{blue}{\xi}}-{\textcolor{blue}{\xi}}_{*}\right|} {\textcolor{blue}{\Psi }}_{{\textcolor{blue}{\xi}}{\textcolor{blue}{\xi}}}^{2}(\tau,{\textcolor{blue}{\xi}}) d{\textcolor{blue}{\xi}} d\tau\\
&+C \int_{0}^{t} e^{\beta\tau} \int_{-\infty}^{+\infty} e^{\alpha\left|{\textcolor{blue}{\xi}}-{\textcolor{blue}{\xi}}_{*}\right|} H_{{\textcolor{blue}{\xi}}{\textcolor{blue}{\xi}}}^{2} (\tau,{\textcolor{blue}{\xi}}) d{\textcolor{blue}{\xi}} d\tau\\
\leq& C \int_{0}^{t} e^{\beta\tau} \int_{-\infty}^{+\infty} e^{\alpha\left|{\textcolor{blue}{\xi}}-{\textcolor{blue}{\xi}}_{*}\right|} {\textcolor{blue}{\Psi }}_{{\textcolor{blue}{\xi}}{\textcolor{blue}{\xi}}}^{2}(\tau,{\textcolor{blue}{\xi}}) d{\textcolor{blue}{\xi}} d\tau+C {\textcolor{red}{\delta}}
\end{aligned}
\end{equation}
using Lemma \ref{yl6}. Substituting \eqref{s4.20}-\eqref{s4.23} into \eqref{s4.19},  choosing a sufficiently small constant $\varepsilon_{2}$ with the aid of Lemma \ref{yl9} and  Lemma \ref{yl10}, we obtain the proof of Lemma \ref{yl11}.
\end{proof}
\begin{lma}\label{yl12}
\textcolor{blue}{U}nder the same assumptions in Lemma \ref{yl9},  we have the following inequality
\begin{equation}
\begin{aligned}
&e^{\beta t} \int_{-\infty}^{+\infty} e^{\alpha\left|{\textcolor{blue}{\xi}}-{\textcolor{blue}{\xi}}_{*}\right|} {\textcolor{blue}{\Psi }}_{{\textcolor{blue}{\xi}} {\textcolor{blue}{\xi}} {\textcolor{blue}{\xi}}}^{2}(t,{\textcolor{blue}{\xi}}) d{\textcolor{blue}{\xi}}+\int_{0}^{t} e^{\beta\tau} \int_{-\infty}^{+\infty} e^{\alpha\left|{\textcolor{blue}{\xi}}-{\textcolor{blue}{\xi}}_{*}\right|} {\textcolor{blue}{\Psi }}_{{\textcolor{blue}{\xi}}{\textcolor{blue}{\xi}}{\textcolor{blue}{\xi}}{\textcolor{blue}{\xi}}}^{2}(\tau,{\textcolor{blue}{\xi}}) d{\textcolor{blue}{\xi}} d\tau \\
 \leq& C \int_{-\infty}^{+\infty} e^{\alpha\left|{\textcolor{blue}{\xi}}-{\textcolor{blue}{\xi}}_{*}\right|}\left({\textcolor{blue}{\Psi }}_{0}^{2}+{\textcolor{blue}{\Psi }}_{0{\textcolor{blue}{\xi}}}^{2}+{\textcolor{blue}{\Psi }}_{0{\textcolor{blue}{\xi}} {\textcolor{blue}{\xi}}}^{2}+{\textcolor{blue}{\Psi }}_{0 {\textcolor{blue}{\xi}}{\textcolor{blue}{\xi}}{\textcolor{blue}{\xi}}}^{2}\right)({\textcolor{blue}{\xi}}) d{\textcolor{blue}{\xi}}+C {\textcolor{red}{\delta}}.
\end{aligned}
\end{equation}
\end{lma}
\begin{proof}
Differentiating of \eqref{x3.2} with respect to ${\textcolor{blue}{\xi}}$ three times, multiplying the result by $2 e^{\beta t} e^{\alpha\left|{\textcolor{blue}{\xi}}-{\textcolor{blue}{\xi}}_{*}\right|} {\textcolor{blue}{\Psi }}_{{\textcolor{blue}{\xi}}{\textcolor{blue}{\xi}}{\textcolor{blue}{\xi}}}$ and integrating the result with respect to $t,{\textcolor{blue}{\xi}}$ over $[0,t] \times \mathbb{R}$,  we have
\begin{equation}\label{sz4.13}
\begin{aligned}
&e^{\beta t} \int_{-\infty}^{+\infty} e^{\alpha\left|{\textcolor{blue}{\xi}}-{\textcolor{blue}{\xi}}_{*}\right|} {\textcolor{blue}{\Psi }}_{{\textcolor{blue}{\xi}}{\textcolor{blue}{\xi}}{\textcolor{blue}{\xi}}}^{2}(t,{\textcolor{blue}{\xi}}) d{\textcolor{blue}{\xi}}+2 \textcolor{blue}{\gamma} \int_{0}^{t} e^{\beta\tau} \int_{-\infty}^{+\infty} e^{\alpha\left|{\textcolor{blue}{\xi}}-{\textcolor{blue}{\xi}}_{*}\right|} {\textcolor{blue}{\Psi }}_{{\textcolor{blue}{\xi}} {\textcolor{blue}{\xi}}{\textcolor{blue}{\xi}}{\textcolor{blue}{\xi}}}^{2}(\tau,{\textcolor{blue}{\xi}}) d{\textcolor{blue}{\xi}} d\tau\\
\leq& \int_{-\infty}^{+\infty} e^{\alpha\left|{\textcolor{blue}{\xi}}-{\textcolor{blue}{\xi}}_{*}\right|} {\textcolor{blue}{\Psi }}_{0{\textcolor{blue}{\xi}} {\textcolor{blue}{\xi}}{\textcolor{blue}{\xi}}}^{2}({\textcolor{blue}{\xi}}) d{\textcolor{blue}{\xi}}+\int_{0}^{t} e^{\beta\tau} \int_{-\infty}^{+\infty} \sum_{i{\textcolor{red}{=}}1}^7  e^{\alpha\left|{\textcolor{blue}{\xi}}-{\textcolor{blue}{\xi}}_{*}\right|} n_i d{\textcolor{blue}{\xi}} d\tau,
\end{aligned}
\end{equation}
where
\begin{equation*}
\begin{aligned}
n_1{\textcolor{red}{=}}& (\beta+s\alpha)  {\textcolor{blue}{\Psi }}_{{\textcolor{blue}{\xi}}{\textcolor{blue}{\xi}}{\textcolor{blue}{\xi}}}^{2}, &n_2{\textcolor{red}{=}}&3 \alpha\textcolor{blue}{\mu} {\textcolor{blue}{\Psi }}_{{\textcolor{blue}{\xi}} {\textcolor{blue}{\xi}}{\textcolor{blue}{\xi}}{\textcolor{blue}{\xi}}}^{2},\\
n_3{\textcolor{red}{=}}& {2(\alpha^{2}\textcolor{blue}{\mu} +\alpha \textcolor{blue}{\gamma} +|f'\left(\textcolor{blue}{U}+{\textcolor{blue}{\Psi }}_{{\textcolor{blue}{\xi}} }\right)|)\left|{\textcolor{blue}{\Psi }}_{{\textcolor{blue}{\xi}}{\textcolor{blue}{\xi}}{\textcolor{blue}{\xi}}} {\textcolor{blue}{\Psi }}_{{\textcolor{blue}{\xi}}{\textcolor{blue}{\xi}}{\textcolor{blue}{\xi}} {\textcolor{blue}{\xi}}}\right|} ,&
n_4{\textcolor{red}{=}}&2 \left|f''\left({\textcolor{blue}{\xi}}_{3}\right)\textcolor{blue}{U}_{{\textcolor{blue}{\xi}}{\textcolor{blue}{\xi}}} {\textcolor{blue}{\Psi }}_{{\textcolor{blue}{\xi}}} {\textcolor{blue}{\Psi }}_{{\textcolor{blue}{\xi}}{\textcolor{blue}{\xi}}{\textcolor{blue}{\xi}} {\textcolor{blue}{\xi}}}\right|,\\
n_5{\textcolor{red}{=}}&4 \left|f''\left(\textcolor{blue}{U}+{\textcolor{blue}{\Psi }}_{{\textcolor{blue}{\xi}}}\right) \textcolor{blue}{U}_{{\textcolor{blue}{\xi}}}{\textcolor{blue}{\Psi }}_{{\textcolor{blue}{\xi}}{\textcolor{blue}{\xi}}{\textcolor{blue}{\xi}}{\textcolor{blue}{\xi}}} {\textcolor{blue}{\Psi }}_{{\textcolor{blue}{\xi}} {\textcolor{blue}{\xi}}}\right|,&
n_6{\textcolor{red}{=}}&2 \left|f'''\left({\textcolor{blue}{\xi}}_{4}\right) \textcolor{blue}{U}_{{\textcolor{blue}{\xi}}}^{2}{\textcolor{blue}{\Psi }}_{{\textcolor{blue}{\xi}}} {\textcolor{blue}{\Psi }}_{{\textcolor{blue}{\xi}}{\textcolor{blue}{\xi}}{\textcolor{blue}{\xi}}{\textcolor{blue}{\xi}} }\right|,\\
n_7{\textcolor{red}{=}}&2 \left|f''\left(\textcolor{blue}{U}+{\textcolor{blue}{\Psi }}_{{\textcolor{blue}{\xi}}}\right) {\textcolor{blue}{\Psi }}_{{\textcolor{blue}{\xi}}{\textcolor{blue}{\xi}}{\textcolor{blue}{\xi}}{\textcolor{blue}{\xi}}} {\textcolor{blue}{\Psi }}^{2}_{{\textcolor{blue}{\xi}} {\textcolor{blue}{\xi}}}\right|,&
n_8{\textcolor{red}{=}}&2 \left|{\textcolor{blue}{\Psi }}_{{\textcolor{blue}{\xi}}{\textcolor{blue}{\xi}}{\textcolor{blue}{\xi}}} H_{{\textcolor{blue}{\xi}}{\textcolor{blue}{\xi}}{\textcolor{blue}{\xi}}} \right|.
\end{aligned}
\end{equation*}
Now we estimate the last six terms $N_{i}(i{\textcolor{red}{=}}3,4,5,6,7,8)$ on the right-hand side of \eqref{sz4.13}, where $ N_{i}{\textcolor{red}{=}}\int_{0}^{t} e^{\beta\tau} \int_{-\infty}^{+\infty}  e^{\alpha\left|{\textcolor{blue}{\xi}}-{\textcolor{blue}{\xi}}_{*}\right|} n_i d{\textcolor{blue}{\xi}} d\tau$. With the aid of Cauchy inequality,  for $i{\textcolor{red}{=}}3,4,5,6$, we have
\begin{equation}\label{sz4.14}
\begin{aligned}
N_{i} \leq&  \varepsilon_{3} \int_{0}^{t} e^{\beta\tau} \int_{-\infty}^{+\infty} e^{\alpha\left|{\textcolor{blue}{\xi}}-{\textcolor{blue}{\xi}}_{*}\right|} {\textcolor{blue}{\Psi }}_{{\textcolor{blue}{\xi}} {\textcolor{blue}{\xi}}{\textcolor{blue}{\xi}} {\textcolor{blue}{\xi}}}^{2}(\tau,{\textcolor{blue}{\xi}}) d{\textcolor{blue}{\xi}} d\tau\\
&+C_{\varepsilon_{3}} \int_{0}^{t} e^{\beta\tau} \int_{-\infty}^{+\infty} e^{\alpha\left|{\textcolor{blue}{\xi}}-{\textcolor{blue}{\xi}}_{*}\right|} ({\textcolor{blue}{\Psi }}_{{\textcolor{blue}{\xi}}}+{\textcolor{blue}{\Psi }}_{{\textcolor{blue}{\xi}}{\textcolor{blue}{\xi}}}+{\textcolor{blue}{\Psi }}_{{\textcolor{blue}{\xi}}{\textcolor{blue}{\xi}}{\textcolor{blue}{\xi}}})^{2}(\tau,{\textcolor{blue}{\xi}}) d{\textcolor{blue}{\xi}} d\tau.
\end{aligned}
\end{equation}
For $N_7$, with the help of Sobolev inequality, we have
\begin{equation}
\begin{aligned}
N_{7}\leq & C \int_{0}^{t} e^{\beta\tau} \|{\textcolor{blue}{\Psi }}\|_{3} \int_{-\infty}^{+\infty} e^{\alpha\left|{\textcolor{blue}{\xi}}-{\textcolor{blue}{\xi}}_{*}\right|}\left| {\textcolor{blue}{\Psi }}_{{\textcolor{blue}{\xi}}{\textcolor{blue}{\xi}}{\textcolor{blue}{\xi}}{\textcolor{blue}{\xi}}} {\textcolor{blue}{\Psi }}_{{\textcolor{blue}{\xi}} {\textcolor{blue}{\xi}}}\right|(\tau,{\textcolor{blue}{\xi}}) d{\textcolor{blue}{\xi}} d\tau\\
\leq&  \varepsilon_{3} \int_{0}^{t} e^{\beta\tau} \int_{-\infty}^{+\infty} e^{\alpha\left|{\textcolor{blue}{\xi}}-{\textcolor{blue}{\xi}}_{*}\right|} {\textcolor{blue}{\Psi }}_{{\textcolor{blue}{\xi}} {\textcolor{blue}{\xi}}{\textcolor{blue}{\xi}} {\textcolor{blue}{\xi}}}^{2}(\tau,{\textcolor{blue}{\xi}}) d{\textcolor{blue}{\xi}} d\tau\\
&+C_{\varepsilon_{3}} \int_{0}^{t} e^{\beta\tau} \int_{-\infty}^{+\infty} e^{\alpha\left|{\textcolor{blue}{\xi}}-{\textcolor{blue}{\xi}}_{*}\right|}  {\textcolor{blue}{\Psi }}_{{\textcolor{blue}{\xi}}{\textcolor{blue}{\xi}}}^{2}(\tau,{\textcolor{blue}{\xi}}) d{\textcolor{blue}{\xi}} d\tau.
\end{aligned}
\end{equation}
\textcolor{blue}{U}sing Lemma \ref{yl6}, similar to \eqref{s4.6}, we have
\begin{equation}\label{sz4.16}
\begin{aligned}
N_{8} \leq & C \int_{0}^{t} e^{\beta\tau} \int_{-\infty}^{+\infty} e^{\alpha\left|{\textcolor{blue}{\xi}}-{\textcolor{blue}{\xi}}_{*}\right|} {\textcolor{blue}{\Psi }}_{{\textcolor{blue}{\xi}}{\textcolor{blue}{\xi}}{\textcolor{blue}{\xi}}}^{2}(\tau,{\textcolor{blue}{\xi}}) d{\textcolor{blue}{\xi}} d\tau\\
&+C  \int_{0}^{t} e^{\beta\tau} \int_{-\infty}^{+\infty} e^{\alpha\left|{\textcolor{blue}{\xi}}-{\textcolor{blue}{\xi}}_{*}\right|} H_{{\textcolor{blue}{\xi}}{\textcolor{blue}{\xi}}{\textcolor{blue}{\xi}}}^{2}(\tau,{\textcolor{blue}{\xi}}) d{\textcolor{blue}{\xi}} d\tau\\
 \leq & C \int_{0}^{t} e^{\beta\tau} \int_{-\infty}^{+\infty} e^{\alpha\left|{\textcolor{blue}{\xi}}-{\textcolor{blue}{\xi}}_{*}\right|} {\textcolor{blue}{\Psi }}_{{\textcolor{blue}{\xi}}{\textcolor{blue}{\xi}}{\textcolor{blue}{\xi}}}^{2}(\tau,{\textcolor{blue}{\xi}}) d{\textcolor{blue}{\xi}} d\tau+C  {\textcolor{red}{\delta}}.
\end{aligned}
\end{equation}
Substituting $\eqref{sz4.14}$-$\eqref{sz4.16}$ into $\eqref{sz4.13}$, choosing a sufficiently small constant $\varepsilon_{3}$,  with the aid of Lemma \ref{yl9}-Lemma \ref{yl11}, we obtain the proof of Lemma \ref{yl12}.
\end{proof}
\section{Proof of the Main Results}\label{section5}
\subsection{Proof of Theorem \ref{theorem201}  }
Taking $\alpha{\textcolor{red}{=}}\beta{\textcolor{red}{=}}0$ in Lemma \ref{yl9}-Lemma \ref{yl12}, one gets Proposition  \ref{pp3}. Making full use of   Proposition \ref{pp3}, we can extend the unique local solution ${\textcolor{blue}{\Psi }}$ to $T {\textcolor{red}{=}}+\infty$ by the standard continuation process.
As long as Proposition \ref{pp3} is proved, we can extend the unique local solution $ {\textcolor{blue}{\Psi }} $ to $T {\textcolor{red}{=}}+\infty $ by the standard continuation process. We have the following lemma.
\begin{lma}\label{yl7} Suppose ${\textcolor{blue}{\Psi }}_{0} \in H^{3}$ there exists a positive constant $\epsilon_{1}{\textcolor{red}{=}}\frac{\epsilon_{0}}{\sqrt{C_{1}}}$, such that if
$$ \left\|{\textcolor{blue}{\Psi }}_{0}\right\|_{3}^{2} + {\textcolor{red}{\delta}}\leq \epsilon_{1}^{2},$$ then the initial  problem \eqref{x3.2}, \eqref{x2017} has a unique global solution
\begin{equation}
\sup_{t\geq0}\|{\textcolor{blue}{\Psi }}(t)\|_{3}^{2}+\int_{0}^{\infty} \left\|{\textcolor{blue}{\Psi }}_{{\textcolor{blue}{\xi}}}(\tau)\right\|_{3}^{2}
d\tau\leq  C_{1}\left(\left\|{\textcolor{blue}{\Psi }}_{0}\right\|_{3}^{2}+{\textcolor{red}{\delta}}\right).
\end{equation}
\end{lma}
Combining Lemma \ref{yl5} and Lemma \ref{yl7}, we complete the proof of Theorem  \ref{theorem201}.
\subsection{Proof of Theorem \ref{theorem2.1}}
Taking  $\alpha>0,\beta>0,$ in Lemma \ref{yl9}-Lemma \ref{yl12}, we have
\begin{lma}\label{yl8}Suppose that ${\textcolor{blue}{\Psi }}(t,{\textcolor{blue}{\xi}})$  is a global smooth solution to the Cauchy problem \eqref{x3.2},\eqref{x2017}. If \eqref{x2.17} holds, there exists a positive constant $\epsilon_{2}<\epsilon_{1}$, such that if $ \left\|{\textcolor{blue}{\Psi }}_0 \right\|_{3}^{2} + {\textcolor{red}{\delta}}\leq \epsilon_{2}^{2}$, we have
$$
\|{\textcolor{blue}{\Psi }}(t)\|_{H^{3}}^{2} \leq C_{2} e^{-\beta t}\left\{\int_{-\infty}^{+\infty} e^{\alpha\left|{\textcolor{blue}{\xi}}-{\textcolor{blue}{\xi}}_{*}\right|}\left({\textcolor{blue}{\Psi }}_{0}^{2}+{\textcolor{blue}{\Psi }}_{0{\textcolor{blue}{\xi}}}^{2}+{\textcolor{blue}{\Psi }}_{0{\textcolor{blue}{\xi}} {\textcolor{blue}{\xi}}}^{2}+{\textcolor{blue}{\Psi }}_{0{\textcolor{blue}{\xi}}{\textcolor{blue}{\xi}}{\textcolor{blue}{\xi}}}^{2}\right)({\textcolor{blue}{\xi}}) d{\textcolor{blue}{\xi}}+{\textcolor{red}{\delta}}\right\}.
$$
Here $C_{2}>0$ is a constant, $\alpha $ and $\beta $ are two positive constants,  satisfying \eqref{x2.18}.
\end{lma}

Theorem  \ref{theorem201} gives that ${\textcolor{blue}{\Psi }}(t,{\textcolor{blue}{\xi}})$ is a global smooth solution. Combining $\theta>\beta$, Lemma \ref{yl5} and Lemma \ref{yl8}, we complete the proof of Theorem  \ref{theorem2.1}.
\section{Proof of Lemmas \ref{yl4}, \ref{yl6}}\label{section6}

For convenience, we define
\begin{equation}
\begin{aligned}
& w_{l}\textcolor{blue}{({\textcolor{blue}{\xi}},t)}:{\textcolor{red}{=}} v_{l}\textcolor{blue}{({\textcolor{blue}{\xi}},t)}-{\bar{v}}_{l},\quad w_{r}\textcolor{blue}{({\textcolor{blue}{\xi}},t)}:{\textcolor{red}{=}} u_{r}\textcolor{blue}{({\textcolor{blue}{\xi}},t)}- {\bar{v}}_{r}.
\end{aligned}
\end{equation}
\subsection{ Proof of Lemma \ref{yl4}}
\begin{proof}
\textcolor{blue}{U}sing the similar method in \cite{XYY2019}, we obtain that there exists a unique  $ {\textcolor{blue}{\eta }}_0$, such that  the initial data satisfies \eqref{x2.13}. With the aid of Rankine-Hugoniot condition \eqref{x2.3} and Lemma \ref{yl2}, one can easily prove that
\begin{equation}\label{A7.1}
  | {\textcolor{blue}{\eta }}'(t)| \leq C {\textcolor{red}{\delta}} e^{-\theta t}.
\end{equation}
Once  \eqref{A7.1} is proved, one gets that there exists a constant $ {\textcolor{blue}{\eta }}_{\infty}$, such that
\begin{equation}
 {\textcolor{blue}{\eta }}(t)-  {\textcolor{blue}{\eta }}_{\infty}{\textcolor{red}{=}}-\int_{t}^{\infty} {\textcolor{blue}{\eta }}'(t){d}t.
\end{equation}
However, the exact expression of $ {\textcolor{blue}{\eta }}_{\infty} $ is not easy to obtain. Motivated by \cite{XYY2021}, now we find this constant. For $y \in(0,1), N\in\mathbb{N}^{*}$, we define the domain
\begin{equation*}
\Omega_{y}^{N}:{\textcolor{red}{=}}\left\{({\textcolor{blue}{\xi}}, {\tau}):  {\textcolor{blue}{\eta }}(\tau)+(-N+y) p_{l} \leq {\textcolor{blue}{\xi}} \leq  {\textcolor{blue}{\eta }}(\tau)+(N+y) p_{r}, 0 \leq \tau \leq t\right\}.
\end{equation*}
We define $G(z):{\textcolor{red}{=}}f(z) -sz$. With the aid of the equations of $u_{l}$ and $u_{r}$, we have
\begin{equation*}
\iint_{\Omega_{y}^{N}}\mathfrak{E}({\textcolor{blue}{\xi}},\tau) d{\textcolor{blue}{\xi}} d\tau{\textcolor{red}{=}}0,
\end{equation*}
where
\begin{align}
\begin{split}
\mathfrak{E}({\textcolor{blue}{\xi}},\tau){\textcolor{red}{=}}& \left(\partial_{t} u_{l}+\partial_{{\textcolor{blue}{\xi}}} G\left(u_{l}\right)  +  \textcolor{blue}{\mu} \partial_{{\textcolor{blue}{\xi}}}^{3} u_{l}-\textcolor{blue}{\gamma} \partial_{{\textcolor{blue}{\xi}}}^{2} u_{l}\right) g_ {\textcolor{blue}{\eta }} \\          &+\left(\partial_{t} u_{r}+\partial_{{\textcolor{blue}{\xi}}} G\left(u_{r}\right)+ \textcolor{blue}{\mu} \partial_{{\textcolor{blue}{\xi}}}^{3} u_{r}-\textcolor{blue}{\gamma} \partial_{{\textcolor{blue}{\xi}}} ^{2} u_{r}\right)\left(1-g_ {\textcolor{blue}{\eta }} \right)
\end{split}&
\end{align}
Then integrating by parts, and using Green formula, we have
\begin{align}\label{5.31}
\begin{split}
  &\iint_{\Omega_{y}^{N}} \mathfrak{D}({\textcolor{blue}{\xi}},\tau) g_ {\textcolor{blue}{\eta }}'
d{\textcolor{blue}{\xi}} d\tau\\
{\textcolor{red}{=}}&A^{N}\textcolor{blue}{(y,t)}-A^{N}(y,0)-B_{l}^{N}\textcolor{blue}{(y,t)}+B_{r}^{N}\textcolor{blue}{(y,t)},
\end{split}&
\end{align}
where
\begin{align*}
\begin{split}
\mathfrak{D}({\textcolor{blue}{\xi}},\tau):{\textcolor{red}{=}}- {\textcolor{blue}{\eta }} '(t)\left(u_{l}-u_{r}\right)+\left(G\left(u_{l}\right)-\textcolor{blue}{\gamma} \partial_{{\textcolor{blue}{\xi}}} u_{l} +\textcolor{blue}{\mu}\partial^{2}_{{\textcolor{blue}{\xi}}} u_{l} \right)-\left(G\left(u_{r}\right)-\textcolor{blue}{\gamma} \partial_{{\textcolor{blue}{\xi}}}u_{r} +\textcolor{blue}{\mu}\partial^{2}_{{\textcolor{blue}{\xi}}} u_{r}\right),
\end{split}&
\end{align*}
\begin{align*}
\begin{split}
A^{N}\textcolor{blue}{(y,t)}:{\textcolor{red}{=}}&\int_{ {\textcolor{blue}{\eta }}(t)+(-N+y) p_{l}}^{ {\textcolor{blue}{\eta }}(t)+(N+y) p_{r}}\left[u_{l}\textcolor{blue}{({\textcolor{blue}{\xi}},t)} g_ {\textcolor{blue}{\eta }}({\textcolor{blue}{\xi}})+u_{r}({\textcolor{blue}{\xi}}, t)\left(1-g_ {\textcolor{blue}{\eta }}({\textcolor{blue}{\xi}})\right)\right]
d{\textcolor{blue}{\xi}},\\
A^{N}(y,0):{\textcolor{red}{=}}&\int_{ {\textcolor{blue}{\eta }}_{0}+(-N+y) p_{l}}^{ {\textcolor{blue}{\eta }}_{0}+(N+y) p_{r}}\left[u_{l}\textcolor{blue}{({\textcolor{blue}{\xi}},0)} g_{ {\textcolor{blue}{\eta }}_{0}}({\textcolor{blue}{\xi}})+u_{r}({\textcolor{blue}{\xi}}, 0)\left(1-g_{ {\textcolor{blue}{\eta }}_{0}}({\textcolor{blue}{\xi}})\right)\right]
{ d{\textcolor{blue}{\xi}} },\\
B_{l}^{N}\textcolor{blue}{(y,t)}:{\textcolor{red}{=}}&\int_{0}^{t}\Big\{\left(G\left(u_{l}\right)-\textcolor{blue}{\gamma} \partial_{{\textcolor{blue}{\xi}}} u_{l} +\textcolor{blue}{\mu}\partial^{2}_{{\textcolor{blue}{\xi}}} u_{l} \right) g_ {\textcolor{blue}{\eta }}+\left(G\left(u_{r}\right)-\textcolor{blue}{\gamma} \partial_{{\textcolor{blue}{\xi}}} u_{r} +\textcolor{blue}{\mu}\partial^{2}_{{\textcolor{blue}{\xi}}} u_{r} \right)\left(1-g_ {\textcolor{blue}{\eta }}\right) \\
&- {\textcolor{blue}{\eta }}'(\tau)\left[u_{l} g_ {\textcolor{blue}{\eta }}+u_{r}\left(1-g_ {\textcolor{blue}{\eta }}\right)\right]\Big\}\left( {\textcolor{blue}{\eta }}(\tau)+(-N+y) p_{ l}, \tau\right)
d\tau,\\
B_{r}^{N}\textcolor{blue}{(y,t)}:{\textcolor{red}{=}} &\int_{0}^{t}\Big\{\left(G\left(u_{l}\right)-\textcolor{blue}{\gamma} \partial_{{\textcolor{blue}{\xi}}} u_{l} +\textcolor{blue}{\mu}\partial^{2}_{{\textcolor{blue}{\xi}}} u_{l} \right) g_ {\textcolor{blue}{\eta }}+\left(G\left(u_{r}\right)-\textcolor{blue}{\gamma} \partial_{{\textcolor{blue}{\xi}}} u_{r} +\textcolor{blue}{\mu}\partial^{2}_{{\textcolor{blue}{\xi}}} u_{r} \right)\left(1-g_ {\textcolor{blue}{\eta }}\right)   \\
& -  {\textcolor{blue}{\eta }}'(\tau)\left[u_{l} g_ {\textcolor{blue}{\eta }}+u_{r}\left(1-g_ {\textcolor{blue}{\eta }}\right)\right]\Big\}\left( {\textcolor{blue}{\eta }}(\tau)+(-N+y) p_{ r},  \tau\right)d\tau.
\end{split}&
\end{align*}
With the aid of \eqref{x2.12}, we have
\begin{align*}
 \iint_{\Omega_{y}^{N}}\mathfrak{D}({\textcolor{blue}{\xi}},\tau) g_ {\textcolor{blue}{\eta }}'d{\textcolor{blue}{\xi}} d\tau{\textcolor{red}{=}}0 \quad \text {as} N \rightarrow+\infty.
\end{align*}
(i) The integrals on $\{\tau{\textcolor{red}{=}}0\}$ and $\{\tau{\textcolor{red}{=}}t\}$.
 With the help of Lemma \ref{yl2}, we have $\left\|w_{l}\right\|_{L^{\infty}}+\left\|w_{r}\right\|_{L^{\infty}} \leq C {\textcolor{red}{\delta}} e^{-\alpha t} $.  With the aid of \eqref{1.3},  one gets
\begin{align}
\begin{split}
 J^{N}\textcolor{blue}{(y,t)}:{\textcolor{red}{=}}& A^{N}\textcolor{blue}{(y,t)}-A^{N}(y,0) \\
{\textcolor{red}{=}}&\int_{(-N+y) p_{l}}^{(N+y) p_{r}}\left[w_{l}({\textcolor{blue}{\xi}}+ {\textcolor{blue}{\eta }}(t),t) g({\textcolor{blue}{\xi}})+w_{r}({\textcolor{blue}{\xi}}+ {\textcolor{blue}{\eta }}(t), t)(1-g({\textcolor{blue}{\xi}}))\right]
{d{\textcolor{blue}{\xi}} }\\
&-\int_{ {\textcolor{blue}{\eta }}_{0}+(-N+y) p_{l}}^{ {\textcolor{blue}{\eta }}_{0}+(N+y) p_{r}}\left[w_{0l}({\textcolor{blue}{\xi}}) g\left({\textcolor{blue}{\xi}}- {\textcolor{blue}{\eta }}_{0}\right)+w_{0r}({\textcolor{blue}{\xi}})\left(1-g\left({\textcolor{blue}{\xi}}- {\textcolor{blue}{\eta }}_{0}\right)\right)\right]d{\textcolor{blue}{\xi}}\\
\leq &C {\textcolor{red}{\delta}}\left(e^{-\alpha t}\right)+\int_{ {\textcolor{blue}{\eta }}_{0}+(-N+y) p_{l}}^{ {\textcolor{blue}{\eta }}_{0}+y p_{l}}\left(w_{0l}-w_{0r}\right)({\textcolor{blue}{\xi}})\left(1-g\left({\textcolor{blue}{\xi}}- {\textcolor{blue}{\eta }}_{0}\right)\right)d{\textcolor{blue}{\xi}}\\
&-\int_{ {\textcolor{blue}{\eta }}_{0}+y p_{l}}^{ {\textcolor{blue}{\eta }}_{0}+y p_{r}}\left[w_{0l}({\textcolor{blue}{\xi}}) g\left({\textcolor{blue}{\xi}}- {\textcolor{blue}{\eta }}_{0}\right)+w_{0 r}({\textcolor{blue}{\xi}})\left(1-g\left({\textcolor{blue}{\xi}}- {\textcolor{blue}{\eta }}_{0}\right)\right)\right]d{\textcolor{blue}{\xi}} \\
&-\int_{ {\textcolor{blue}{\eta }}_{0}+y p_{r}}^{ {\textcolor{blue}{\eta }}_{0}+(N+y) p_{r}}\left(w_{0 l}-w_{0 r}\right)({\textcolor{blue}{\xi}}) g\left({\textcolor{blue}{\xi}}- {\textcolor{blue}{\eta }}_{0}\right)d{\textcolor{blue}{\xi}}.
\end{split}&
\end{align}
Thus, we have
\begin{align}\label{5.35}
\begin{split}
J(y, t)\leq &  C {\textcolor{red}{\delta}} \left(e^{-\alpha t}\right)+\int_{-\infty}^{ {\textcolor{blue}{\eta }}_{0}+y p_{l}}\left(w_{0l}-w_{0 r}\right)({\textcolor{blue}{\xi}})\left(1-g\left({\textcolor{blue}{\xi}}- {\textcolor{blue}{\eta }}_{0}\right)\right)d{\textcolor{blue}{\xi}}\\
&-\int_{ {\textcolor{blue}{\eta }}_{0}+y p_{l}}^{ {\textcolor{blue}{\eta }}_{0}+y p_{r}}\left[w_{0l}({\textcolor{blue}{\xi}}) g\left({\textcolor{blue}{\xi}}- {\textcolor{blue}{\eta }}_{0}\right)+w_{0r}({\textcolor{blue}{\xi}})\left(1-g\left({\textcolor{blue}{\xi}}- {\textcolor{blue}{\eta }}_{0}\right)\right)\right]d{\textcolor{blue}{\xi}}\\
&-\int_{ {\textcolor{blue}{\eta }}_{0}+y p_{r}}^{+\infty}\left(w_{0 l}-w_{0 r}\right)({\textcolor{blue}{\xi}}) g\left({\textcolor{blue}{\xi}}- {\textcolor{blue}{\eta }}_{0}\right)d{\textcolor{blue}{\xi}},
\end{split}&
\end{align}
where $J(y, t):{\textcolor{red}{=}} \lim_{N \rightarrow \infty} J^{N}\textcolor{blue}{(y,t)}$. With the aid of $\eqref{x2.10},$ it follows that
\begin{align}
\begin{split}
0{\textcolor{red}{=}}&-\int_{ -\infty}^{\infty}\left[u_{0}({\textcolor{blue}{\xi}})-\textcolor{blue}{\phi}\left({\textcolor{blue}{\xi}}- {\textcolor{blue}{\eta }}_{0}\right)-w_ {0l}({\textcolor{blue}{\xi}}) g\left({\textcolor{blue}{\xi}}- {\textcolor{blue}{\eta }}_{0}\right)-w_{0 r}({\textcolor{blue}{\xi}})\left(1-g\left({\textcolor{blue}{\xi}}- {\textcolor{blue}{\eta }}_{0}\right)\right)\right]d{\textcolor{blue}{\xi}}\\
{\textcolor{red}{=}}&-\int_{-\infty}^{0}\left(u_{0}-\textcolor{blue}{\phi}-w_{0l}\right)({\textcolor{blue}{\xi}}) d{\textcolor{blue}{\xi}}-\int_{0}^{+\infty}\left(u_{0}-\textcolor{blue}{\phi}-w_{0r}\right)({\textcolor{blue}{\xi}})d{\textcolor{blue}{\xi}}\\
&+\left(\bar{u}_{l}-\bar{u}_{r}\right)  {\textcolor{blue}{\eta }}_{0}-\int_{-\infty}^{0}\left(w_{0l}-w_{0 r}\right)({\textcolor{blue}{\xi}})\left(1-g\left({\textcolor{blue}{\xi}}- {\textcolor{blue}{\eta }}_{0}\right)\right)d{\textcolor{blue}{\xi}}\\
&+\int_{0}^{+\infty}\left(w_{0 l}-w_{0 r}\right)({\textcolor{blue}{\xi}}) g\left({\textcolor{blue}{\xi}}- {\textcolor{blue}{\eta }}_{0}\right)d{\textcolor{blue}{\xi}}.
\end{split}&
\end{align}
Together with \eqref{5.35}, we have
\begin{align}
\begin{split}
J\textcolor{blue}{(y,t)}\leq & C {\textcolor{red}{\delta}}\left(e^{-\alpha t}\right) +\left(\bar{u}_{l}-\bar{u}_{r}\right)  {\textcolor{blue}{\eta }}_{0}\\
&-\int_{-\infty}^{0}\left(u_{0}-\textcolor{blue}{\phi}-w_{0l}\right)({\textcolor{blue}{\xi}})d{\textcolor{blue}{\xi}}-\int_{0}^{+\infty}\left(u_{0}-\textcolor{blue}{\phi}-w_{0r}\right)({\textcolor{blue}{\xi}})
d{\textcolor{blue}{\xi}}\\
&+\int_{0}^{ {\textcolor{blue}{\eta }}_{0}+y p_{i}} w_{0l}({\textcolor{blue}{\xi}})d{\textcolor{blue}{\xi}}-\int_{0}^{ {\textcolor{blue}{\eta }}_{0}+y p_{r}} w_{0 r}({\textcolor{blue}{\xi}})
d{\textcolor{blue}{\xi}}.
\end{split}&
\end{align}
Since $\int_{0}^{p_{i}}w_{0i}d{\textcolor{blue}{\xi}}{\textcolor{red}{=}}0$  for $i{\textcolor{red}{=}}l, r,$  $\int_{0}^{y} w_{0i}({\textcolor{blue}{\xi}})
d{\textcolor{blue}{\xi}}$ is periodic with respective to $y$ with period $p_{i}$. Therefore
\begin{align}
\begin{split}
\int_{0}^{1}\int_{0}^{ {\textcolor{blue}{\eta }}_{0}+y p_{i}} w_{0 i}({\textcolor{blue}{\xi}}) d{\textcolor{blue}{\xi}} dy &{\textcolor{red}{=}}\frac{1}{p_{i}} \int_{0}^{p_{i}} \int_{0}^{ {\textcolor{blue}{\eta }}_{0}+y} w_{0 i}({\textcolor{blue}{\xi}})d{\textcolor{blue}{\xi}} dy\\
&{\textcolor{red}{=}}\frac{1}{p_{i}} \int_{0}^{p_{i}} \int_{0}^{y} w_{0 i}({\textcolor{blue}{\xi}})d{\textcolor{blue}{\xi}} dy.
\end{split}&
\end{align}
So, we have
\begin{align}
\begin{split}
\int_{0}^{1} J\textcolor{blue}{(y,t)} dy\leq & C {\textcolor{red}{\delta}}\left(e^{-\alpha t}\right) +\left(\bar{u}_{l}-\bar{u}_{r}\right)  {\textcolor{blue}{\eta }}_{0} \\
&-\int_{-\infty}^{0}\left(u_{0}-\textcolor{blue}{\phi}-w_{0l}\right)({\textcolor{blue}{\xi}})d{\textcolor{blue}{\xi}}-\int_{0}^{+\infty}\left(u_{0}-\textcolor{blue}{\phi}-w_{0r}\right)({\textcolor{blue}{\xi}})d{\textcolor{blue}{\xi}}\\
&+\frac{1}{p_{l}} \int_{0}^{p_{l}} \int_{0}^{y} w_{0 l}({\textcolor{blue}{\xi}})d{\textcolor{blue}{\xi}} dy-\frac{1}{p_{r}} \int_{0}^{p_{r}} \int_{0}^{y} w_{0r}({\textcolor{blue}{\xi}})d{\textcolor{blue}{\xi}} dy.
\end{split}&
\end{align}
(ii) The integrals on two sides.
Since $u_{l}$ is periodic, it holds that
\begin{align}\label{5.40}
\begin{split}
B_{l}^{N}(y, t){\textcolor{red}{=}}&\int_{0}^{t}\left\{\left(G\left(u_{l}\right)-\textcolor{blue}{\gamma} \partial_{{\textcolor{blue}{\xi}}} u_{l} +  \textcolor{blue}{\mu} \partial^{2}_{{\textcolor{blue}{\xi}}} u_{l} \right)   \left(  {\textcolor{blue}{\eta }}(\tau)+y p_{l}, \tau\right) g\left((-N+y) p_{l}\right)\right.\\
&- {\textcolor{blue}{\eta }}'(\tau) u_{l}\left( {\textcolor{blue}{\eta }}(\tau)+y p_{l}, \tau\right)\\
&\left.+[\cdots]\left(1-g_ {\textcolor{blue}{\eta }}\right)\left( {\textcolor{blue}{\eta }}(\tau)+(-N+y) p_{l}, \tau\right)\right\}
{d \tau }.\\
\end{split}&
\end{align}
where $[\cdots]$ denotes the remaining terms which are bounded. Then by taking the limit $N \rightarrow+\infty$ in \eqref{5.40} and using Lemma $\ref{yl3},$ one can get
\begin{align}
\begin{split}
\lim _{N \rightarrow+\infty} \int_{0}^{1} B_{l}^{N}\textcolor{blue}{(y,t)}dy&{\textcolor{red}{=}}\int_{0}^{t} \frac{1}{p_{l}} \int_{0}^{p_l} G\left(u_{l}\right)({\textcolor{blue}{\xi}},\tau)
d{\textcolor{blue}{\xi}} d\tau-\bar{u}_{l}\left( {\textcolor{blue}{\eta }} (t)- {\textcolor{blue}{\eta }}_{0}\right). \\
\end{split}&
\end{align}
Similarly, it holds that
\begin{align}
\begin{split}
\lim _{N \rightarrow+\infty} \int_{0}^{1} B_{r}^{N}\textcolor{blue}{(y,t)}dy{\textcolor{red}{=}}\int_{0}^{t} \frac{1}{p_{r}} \int_{0}^{p_r} G\left(u_{r}\right)({\textcolor{blue}{\xi}},\tau)d{\textcolor{blue}{\xi}} d\tau-\bar{u}_{r}\left( {\textcolor{blue}{\eta }} (t)- {\textcolor{blue}{\eta }}_{0}\right).
\end{split}&
\end{align}
Now, with the calculations in (i) and (ii), one can integrate the equation \eqref{5.31} with respect to $y$ over $(0,1),$ and then let $N \rightarrow+\infty,$  for any $t>0$
\begin{equation}
\begin{aligned}
&\int_{0}^{1} J\textcolor{blue}{(y,t)}dy+\left(\bar{u}_{l}-\bar{u}_{r}\right)\left( {\textcolor{blue}{\eta }}(t)- {\textcolor{blue}{\eta }}_{0}\right) \\
{\textcolor{red}{=}}& \int_{0}^{t}\left[\frac{1}{p_{l}} \int_{0}^{p_{l}} G\left(u_{l}\right)({\textcolor{blue}{\xi}},\tau)
d{\textcolor{blue}{\xi}}-\frac{1}{p_{r}} \int_{0}^{p_{r}} G\left(u_{r}\right)({\textcolor{blue}{\xi}},\tau)
d{\textcolor{blue}{\xi}}\right]d\tau.
\end{aligned}
\end{equation}
Note also that for $i{\textcolor{red}{=}}l, r,$
\begin{equation}
\begin{aligned}
\int_{0}^{t} \frac{1}{p_{i}} \int_{0}^{p_{i}} G\left(u_{i}\right)dy d\tau{\textcolor{red}{=}}&\int_{0}^{t} \frac{1}{p_{i}} \int_{0}^{p_{i}}\left(G\left(u_{i}\right)-G\left(\bar{u}_{i}\right)\right)dy d\tau+G\left(\bar{u}_{i}\right)t\\
\leq &\int_{0}^{+\infty} \frac{1}{p_{i}} \int_{0}^{p_{i}}\left(G\left(u_{i}\right)-G\left(\bar{u}_{i}\right)\right)
dy d\tau\\
&+C{\textcolor{red}{\delta}}\left(e^{-\alpha t}\right)+G\left(\bar{u}_{i}\right)t.
\end{aligned}
\end{equation}
Thus we have the proof of Lemma \ref{yl4}.
\end{proof}
\subsection {Proof of Lemma \ref{yl6}}
\begin{proof}
The proof is motivate by \cite{XYY2019}. With the aid of \eqref{x2.9}, when ${\textcolor{blue}{\xi}}< {\textcolor{blue}{\eta }}(t)$, one has
\begin{align*}
\begin{split}
H{\textcolor{red}{=}}&-f(\textcolor{blue}{U})+f(u_{l})g_{ {\textcolor{blue}{\eta }}}
 + f(u_{r}) (1-g_{ {\textcolor{blue}{\eta }}})-3\textcolor{blue}{\mu} (u_{l}-u_{r})_{{\textcolor{blue}{\xi}}} g_ { {\textcolor{blue}{\eta }}}'+ 2\textcolor{blue}{\gamma}(u_{l}-u_{r} )g_{ {\textcolor{blue}{\eta }}}'\\
&-\int_{-\infty}^{{\textcolor{blue}{\xi}}}\mathfrak{N}\textcolor{blue}{({\textcolor{blue}{\xi}},t)}d{\textcolor{blue}{\xi}}:{\textcolor{red}{=}}H_1+H_2,
\end{split}&
\end{align*}
where
\begin{equation*}
\mathfrak{N}\textcolor{blue}{({\textcolor{blue}{\xi}},t)}:{\textcolor{red}{=}} (f(u_{l})-f(u_{r}))g_{ {\textcolor{blue}{\eta }}}'+\textcolor{blue}{\gamma}(u_{l}-u_{r}) g_{ {\textcolor{blue}{\eta }}}''+ \textcolor{blue}{\mu} (u_{l}-u_{r}) g_{ {\textcolor{blue}{\eta }}}'''-(u_{l}-u_{r}) \cdot g_{ {\textcolor{blue}{\eta }}}' \cdot\left(s+ {\textcolor{blue}{\eta }}'(t)\right).
\end{equation*}
\textcolor{blue}{U}sing  Lemma \ref{yl2}--Lemma \ref{yl4}, one gets
\begin{align*}
\begin{split}
H_1{\textcolor{red}{=}}&-f(\textcolor{blue}{U})+f (u_{l})+ \textcolor{blue}{\gamma} (\bar{u}_{l}-\bar u_{r}) g_{ {\textcolor{blue}{\eta }}}'-\textcolor{blue}{\mu}  (\bar{u}_{l}-\bar u_{r})g_{ {\textcolor{blue}{\eta }}}''+\int_{-\infty}^{{\textcolor{blue}{\xi}}} (\bar{u}_{l}-\bar{ u }_{r}) \cdot g_{ {\textcolor{blue}{\eta }}} ' \cdot s d{\textcolor{blue}{\xi}}\\
{\textcolor{red}{=}}&-f(\textcolor{blue}{U})+f (u_{l})+ f(\textcolor{blue}{\phi}_{ {\textcolor{blue}{\eta }}})-f(\bar u_{l})\\
\leq  &C {\textcolor{red}{\delta}} e^{-\theta t}e^{{\sigma_0}({\textcolor{blue}{\xi}}- {\textcolor{blue}{\eta }}(t))},
\end{split}&
\end{align*}
and
\begin{align*}
\begin{split}
H_2{\textcolor{red}{=}}&-\int_{-\infty}^{{\textcolor{blue}{\xi}}} (\bar{u}_{r}-\bar{u}_{l}) \cdot g_{ {\textcolor{blue}{\eta }}}' \cdot  {\textcolor{blue}{\eta }}'(t)d{\textcolor{blue}{\xi}}+ [f(u_{r})-f(\bar{u}_{r}) -f(u_{l})+f(\bar{u}_{l})](1-g_{ {\textcolor{blue}{\eta }}})\\
&+2\textcolor{blue}{\gamma}(w_{l}-w_{r})g_{ {\textcolor{blue}{\eta }}}'-3\textcolor{blue}{\mu}( w_{l}-w_{r})_{{\textcolor{blue}{\xi}}}g_{ {\textcolor{blue}{\eta }}}'\\
&+\int_{-\infty}^{{\textcolor{blue}{\xi}}}(w_{l}-w_{r})\cdot g_{ {\textcolor{blue}{\eta }}}' \cdot\left(s+ {\textcolor{blue}{\eta }}'(t)\right)-[f(u_{l})-f(\bar{u}_{l})+f(\bar{u}_{r})-f(u_{r})]g_{ {\textcolor{blue}{\eta }}}'d{\textcolor{blue}{\xi}}\\
&-\int_{-\infty }^{{\textcolor{blue}{\xi}}}\textcolor{blue}{\gamma}(w_{l}-w_{r})\cdot g_{ {\textcolor{blue}{\eta }}} ''+\textcolor{blue}{\mu}(w_{l}-w_{r})\cdot g_{ {\textcolor{blue}{\eta }}}'''d{\textcolor{blue}{\xi}}\\
\leq& C{\textcolor{red}{\delta}} e^{-\theta t}\left\{\int_{-\infty }^{{\textcolor{blue}{\xi}}}(|g_{ {\textcolor{blue}{\eta }}} '''|+|g_{ {\textcolor{blue}{\eta }}}''|+|g_{ {\textcolor{blue}{\eta }}}'|)d {\textcolor{blue}{\xi}}+(1-g_{ {\textcolor{blue}{\eta }}})\right\}\\
\leq &C {\textcolor{red}{\delta}} e^{-\theta t}e^{\sigma_{0}({\textcolor{blue}{\xi}}- {\textcolor{blue}{\eta }}(t))}.
\end{split}&
\end{align*}
As to ${\textcolor{blue}{\xi}}> {\textcolor{blue}{\eta }}(t), $ using the same method,  we have $ H\textcolor{blue}{({\textcolor{blue}{\xi}},t)}\leq C {\textcolor{red}{\delta}} e^{-\theta t} e^{-{\sigma_0}\left({\textcolor{blue}{\xi}}- {\textcolor{blue}{\eta }}(t)\right)}.$
Similar, it follows that $|H_{{\textcolor{blue}{\xi}}}|,|H_{{\textcolor{blue}{\xi}}{\textcolor{blue}{\xi}}}|,|H_{{\textcolor{blue}{\xi}}{\textcolor{blue}{\xi}}{\textcolor{blue}{\xi}}}|\leq C {\textcolor{red}{\delta}} e^{-\theta t} e^{{\sigma_0}\left|{\textcolor{blue}{\xi}}- {\textcolor{blue}{\eta }}(t)\right|}$.
\end{proof}
\section{Appendix}
\subsection{Proof of Lemma \ref{yl2}}
In this section, we write $\|.\|_{H^{k+1}(\Omega)}$ as $\|.\|_{k+1}$, for convenience.

Claim 1. For any integer $k\geq0$,
\begin{equation}
 \left\|\partial_{{\textcolor{blue}{\xi}}}^{k}(u-\bar{u})\right\|^{2}  +\int_{0}^{t}   \left\|\partial_{{\textcolor{blue}{\xi}}}^{k+ 1} u\right\|^{2} d\tau\leq  C  \left\| (u_{0}-\bar{u})\right\|_{k}^{2}  \quad \forall t>0,
\label{6.1}
\end{equation}
where $C$ is a constant depends only on $u_0, p$ and $f$. Then we prove Claim 1 by the induction method.

\emph{step 1}: We will prove \eqref{6.1} is true when $k{\textcolor{red}{=}}0$. Multiplying  $u-\bar{u}$ on each side of \eqref{s1.1} integrating with respect to ${\textcolor{blue}{\xi}}$ over $[0,p]$, one gets
\begin{equation}
\frac{d}{dt} \left(\left\| u-\bar{u}  \right\|^{2}\right)+2 \textcolor{blue}{\gamma}\left\|  \partial_{{\textcolor{blue}{\xi}}} u \right\|^{2}{\textcolor{red}{=}}0, \quad \forall t>0.
\label{6.2}\end{equation}
Integrating \eqref{6.2} with respect to $t$ over $[0,t]$, we have
\begin{equation}\label{b5}
 \left\| u-\bar{u} \right\|^{2} +2 \textcolor{blue}{\gamma}\int_{0}^{t}  \left\|\partial_{{\textcolor{blue}{\xi}}} u\right\|^{2} d\tau{\textcolor{red}{=}}\left\|u_{0}-\bar{u} \right\|^{2}, \quad \forall t>0.
\end{equation}
\emph{step 2}: We assume \eqref{6.1} is true when $k{\textcolor{red}{=}}2,...,m-1$. We will prove  \eqref{6.1} is true when $k{\textcolor{red}{=}}m$. Taking the derivative $\partial_{{\textcolor{blue}{\xi}}}^{m} $ in \eqref{s1.1},  multiplying $\partial_{{\textcolor{blue}{\xi}}}^{m}u$ on each side, integrating  the result  over $[0,p]$, with the aid of Cauchy inequality, one has
\begin{equation*}
\begin{aligned}
\frac{d}{dt}\left(\left\|\partial_{{\textcolor{blue}{\xi}}}^{m} u\right\|^{2}\right)+2\textcolor{blue}{\gamma}    \left\|\partial_{{\textcolor{blue}{\xi}}}^{m+1} u\right\|^{2} \leq\textcolor{blue}{\gamma} \left\|\partial_{{\textcolor{blue}{\xi}}}^{m+1} u\right\|^{2}+ C_{\textcolor{blue}{\gamma}}\sum_{k{\textcolor{red}{=}}1}^{m}\|\partial^{k}_{{\textcolor{blue}{\xi}}}u\|^2.
\end{aligned}
\end{equation*}
 Thus
\begin{equation}
\begin{aligned}
&\frac{d}{dt}\left(\left\|\partial_{{\textcolor{blue}{\xi}}}^{m} u\right\|^{2}\right)+\textcolor{blue}{\gamma}  \left\|\partial_{{\textcolor{blue}{\xi}}}^{m+1} u\right\|^{2} \leq C \sum_{k{\textcolor{red}{=}}1}^{m}\|\partial^{k}_{{\textcolor{blue}{\xi}}}u\|  ^2.
\end{aligned}
\label{6.3}
\end{equation}
Integrating  \eqref{6.3} over $[0,t]$. Thus for $ \forall t>0$
\begin{equation*}
\begin{aligned}
&\left\|\partial_{{\textcolor{blue}{\xi}}}^{m} u\right\|^{2}+ \textcolor{blue}{\gamma} \int_{0}^{t}  \left\|\partial_{{\textcolor{blue}{\xi}}}^{m+1} u\right\| ^{2}d\tau
\leq  C    \sum_{k{\textcolor{red}{=}}1}^{m }    \int_{0}^{t} \|\partial^{k}_{{\textcolor{blue}{\xi}}}u\|  ^2 d\tau    +\left\|\partial_{{\textcolor{blue}{\xi}}}^{m} u_{0}\right\| ^{2}\\ \leq&   C        \left\|  (u_{0}-\bar{u} )\right\| _{ m}^{2} .
\end{aligned}
\end{equation*}
Thus  Claim 1 is true.

Claim $2.$ For each $k \geq 0, C >0,$ it follows that
\begin{equation}
 \left\| \partial_{{\textcolor{blue}{\xi}}}^{k}(u-\bar{u}) ({\textcolor{blue}{\xi}}, t)  \right \|^{2}   \leq C \left\|   u_{0}-\bar{u}    \right\|_{ {k}}^{2} e^{-\theta t} \quad \forall t \leq 0.
\label{6.4}
\end{equation}

\emph{step 1}:  We will proof \eqref{6.4} is true when $k{\textcolor{red}{=}}0$.
With the aid of   Poincare inequality on $[0, p],$   there exists a constant $\theta>0,$ which depends only on $p,$ such that
\begin{equation}
\int_{0}^{p}\left(\partial_{{\textcolor{blue}{\xi}}} u\right)^{2}({\textcolor{blue}{\xi}}, t) {d{\textcolor{blue}{\xi}}  }  \geq    \frac{\theta}{2\textcolor{blue}{\gamma}} \int_{0}^{p}(u-\bar{u})^{2}({\textcolor{blue}{\xi}}, t) {d{\textcolor{blue}{\xi}}  }.
\label{6.5}\end{equation}
Combining  \eqref{6.2} and  \eqref{6.5}, we have
\begin{equation*}
 \left\|  (u-\bar{u}) ({\textcolor{blue}{\xi}}, t)  \right \|^{2}   \leq C \left\|   (u_{0}-\bar{u})   \right \|^{2} e^{-\theta t} \quad \forall t \leq 0.
\end{equation*}

\emph{step 2}: We assume \eqref{6.4} is true when $  k{\textcolor{red}{=}}1,2,...,m-1$. We will prove  \eqref{6.4} is true when $ k {\textcolor{red}{=}} m$. For any j, $0\leq j\leq  m$, combining \eqref{6.2} and  \eqref{6.3}, we have
\begin{equation}
\frac{d}{{ d t }} \left(\left\|\partial_{{\textcolor{blue}{\xi}}}^{j }(u-\bar{u})\right\|^{2}
\right)+ \textcolor{blue}{\gamma}\left\|\partial_{{\textcolor{blue}{\xi}}}^{j+1} u\right\|^{2}
 \leq  C    \sum_{k{\textcolor{red}{=}}1}^{j }  \|\partial^{k}_{{\textcolor{blue}{\xi}}}u\|  ^2.
\label{6.6}
\end{equation}
Letting $j{\textcolor{red}{=}}m $ in $\eqref{6.6}$,  we have
\begin{equation}
\frac{d}{{ d t }} \left\|\partial_{{\textcolor{blue}{\xi}}}^{m} u\right\|^{2} \leq C_{m} \sum_{k{\textcolor{red}{=}}1}^{m} \left\|\partial_{{\textcolor{blue}{\xi}}}^{k} u\right\|^{2} \leq C \left\|   u_{0}-\bar{u}    \right\|_{m-1}^{2}e^{-\theta t} +C_{m} \left\|\partial_{{\textcolor{blue}{\xi}}}^{m} u\right\|^{2},
\label{6.7}\end{equation}
where $C_{m}>\theta .$ We have used Claim 1 and Sobolev inequality.   
Letting $j{\textcolor{red}{=}} m-1$ in $\eqref{6.6}$, one gets that $\forall t \geq 0$
\begin{equation}
\frac{{ d }}{{ d t }} \left\|\partial_{{\textcolor{blue}{\xi}}}^{m-1}(u-\bar{u})\right\|^{2} +\textcolor{blue}{\gamma} \left\|\partial_{{\textcolor{blue}{\xi}}}^{m} u\right\|^{2}  \leq C\left\|   u_{0}-\bar{u}    \right\|_{m-1}^{2} e^{-\theta t}.
\label{6.8}\end{equation}
Then multiply $\frac{2 C_{m}}{\textcolor{blue}{\gamma}}$ on $\eqref{6.8}$. Add $\eqref{6.7}$ to the result. One gets  $\forall t \geq 0$
\begin{equation}
\frac{d}{{  d t}}\left(\frac{2 C_{m}}{\textcolor{blue}{\gamma}} \left\|\partial_{{\textcolor{blue}{\xi}}}^{m-1}(u-\bar{u})\right\|^{2}
 + \left\|\partial_{{\textcolor{blue}{\xi}}}^{m} u\right\|^{2}  \right)+C_{m} \left\|\partial_{{\textcolor{blue}{\xi}}}^{m} u\right\|^{2}  \leq C e^{-\theta t} \left\|   u_{0}-\bar{u}    \right\|_{m-1}^{2}.
\label{6.9}\end{equation}
Denote
$$E_{m}(t):{\textcolor{red}{=}}\frac{2 C_{m}}{\textcolor{blue}{\gamma}} \left\|\partial_{{\textcolor{blue}{\xi}}}^{m-1}(u-\bar{u})\right\|^{2} + \left\|\partial_{{\textcolor{blue}{\xi}}}^{m} u\right\|^{2}. $$
Taking $k{\textcolor{red}{=}}m-1$ in $\eqref{6.4}$, with the aid of $\eqref{6.9}$,  for $\forall t \geq 0$, we have
\begin{equation}
E_{m}'(t)+C_{m} E_{m}(t)\leq \tilde{C}_{m} \left\|   u_{0}-\bar{u}    \right\|_{m-1}^{2} e^{-\theta t}.
 \label{6.10}\end{equation}
Here $\tilde{C}_{m} $ is a new constant depends on $m$. Multiplying \eqref{6.10} by $e^{C_{m}t}$, we have
\begin{equation*}
\frac{d}{dt}[e^{C_{m}t}E_{m}(t)]\leq \tilde{C}_{m} \left\|u_{0}-\bar{u}    \right\|_{m-1}^{2} e^{(C_{m}-\theta)t}.\\
\end{equation*}
Thus
\begin{equation*}
e^{C_{m}t}E_{m}(t)\leq E_{m}(0)+\frac{\tilde{C}_{m} }{C_{m}-\theta}\left\|   u_{0}-\bar{u}\right\|_{m-1}^{2}[e^{(C_{m}-\theta)t}-1].
\end{equation*}
With the aid of $C_{m}>\theta$, one gets
\begin{equation*}
\begin{aligned}
E_{m}(t) &\leq E_{m}(0) e^{-C_{m}t}+\frac{\tilde{C}_{m} }{C_{m}-\theta}\left\|   u_{0}-\bar{u}\right\|_{m-1}^{2} e^{-\theta t}\\
&\leq C \left\|u_{0}-\bar{u}\right\|_{m}^{2} e^{-\theta t}.
\end{aligned}
\end{equation*}
The proof of Claim 2 is accomplished. Then by Sobolev inequality and Claim $2,$ for any integer $k\leq 0$ and $t\leq 0$, we have
\[
 \left\|  \partial_{{\textcolor{blue}{\xi}}}^{k}(u-\bar{u})\right\|_{L^{\infty}(\mathbb{R})}{\textcolor{red}{=}}\left\|  \partial_{{\textcolor{blue}{\xi}}}^{k}(u-\bar{u})\right\|_{L^{\infty}(\Omega)} \leq C \left\|u_{0} -\bar{u}\right\|_{H^{k+1}(\Omega)} e^{-\theta t}.
 \]
Thus, we finish the proof of Lemma\ref{yl2}.

\end{document}